\documentclass[10pt,twoside,a4paper,reqno,pdftex,final]{amsart}
\usepackage[latin1]{inputenc}

\usepackage[centertags]{amsmath}

\usepackage{amsthm}
\usepackage{amssymb}
\usepackage{amsfonts,a4}
\usepackage{amsxtra}
\usepackage{amscd}
\usepackage{mathabx} 
\usepackage{esint}

\usepackage{tikz}
 \usepackage{pgfplots}

\usepackage[english]{babel}

\usepackage{todonotes}

\usepackage{algorithm}
\usepackage{algorithmic}

\usepackage[mathscr]{eucal}	

\usepackage{comment}

\usepackage{bbm}
\usepackage{enumerate}
\usepackage{color}
\usepackage{ifpdf}
\ifpdf
  \usepackage[bookmarks=true]{hyperref}
  \usepackage{graphicx}
  \DeclareGraphicsExtensions{.pdf,.png,.jpg}
  \graphicspath{{./pic/}}

\else
  \usepackage[xdvi,dvips]{graphicx}
  \DeclareGraphicsExtensions{.ps,.eps}
\fi

\usepackage{graphicx} 
\graphicspath{{./plots/},{./plots/chen_feng_example/astfem/},
  {./plots/chen_feng_example/classic_alg/}}

\usepackage{xspace} %
\usepackage[notcite,notref]{showkeys}

\usepackage[scrtime]{prelim2e}

\newcommand{\dt}{\,{\rm d}t}

\newtheorem{thm}{Theorem}
\newtheorem{cor}[thm]{Corollary}
\newtheorem{lem}[thm]{Lemma}
\newtheorem{prop}[thm]{Proposition}

\newtheorem{rem}[thm]{Remark}
\newtheorem{ass}[thm]{Assumption}

\numberwithin{equation}{section}

\definecolor{applegreen}{rgb}{0.55, 0.71, 0.0}

\newcommand{\ADAPTINIT}[1]{\texttt{ADAPT\_INIT}\ensuremath{{#1}}}
\newcommand{\Am}{\ensuremath{\mathbf{A}}} 
\newcommand{\abs}[1]{\ensuremath{\left|#1\right|}}

\renewcommand{\atop}[2]{\ensuremath{\genfrac{}{}{0pt}{}{#1}{#2}}}

\newcommand{\ASTFEM}{\text{ASTFEM}\xspace}
\DeclareMathOperator*{\argmax}{\arg\!\max}

\newcommand{\B}{\mathcal{B}}
\newcommand{\bilin}[3][]{\ensuremath{\B_{#1}[#2,\,#3]}}

\newcommand{\COARSEN}[1][(U,\grid)]{\texttt{COARSEN}\ensuremath{{#1}}}
\newcommand{\CONSISTENCY}[1][(f,t,\tau)]{\texttt{CONSISTENCY}\ensuremath{{#1}}}

\newcommand{\definedas}{\mathrel{:=}}

\newcommand{\dual}[2]{\ensuremath{\langle #1,\,#2\rangle}}
\newcommand{\Dt}{\ensuremath{\partial _t}}

\newcommand{\dGs}[1][s]{{\rm dG(\ensuremath{#1})}\xspace}
\newcommand{\DUNE}{\textsf{DUNE}\xspace}

\DeclareMathOperator{\divo}{div}

\newcommand{\elm}{\ensuremath{E}\xspace}
\newcommand{\enorm}[2][\Omega]{\left|\negthinspace\left|\negthinspace\left|{#2}%
                      \right|\negthinspace\right|\negthinspace\right|_{#1}}
\newcommand{\Enorm}[2][\Omega]{|\negthinspace|\negthinspace|{#2}%
                      |\negthinspace|\negthinspace|_{#1}}

\newcommand{\est}{\mathcal{E}}

\newcommand{\Est}[3]{\ensuremath{\mathcal{E}_{#1}^{#2}}{(#3)}}
\newcommand{\Einit}[1][v,\grid]{\Est{0}{2}{#1}}
\newcommand{\Econs}[1][f,t,\tau]{\Est{f}{2}{#1}}
\newcommand{\Etime}[1][v^+,v^-,\tau,\grid]{\Est{\tau}{2}{#1}}
\newcommand{\Espace}[1][V,v^-,t,\tau,\bar f,\grid,E]{\Est{\grid}{2}{#1}}
\newcommand{\Ecoarse}[1][v^-, \tau,\grid, E]{\Est{c}{2}{#1}}
\newcommand{\Etc}[1][v^+,v^-,\tau]{\Est{c\tau}{2}{#1}}
\newcommand{\Estar}[1][V,v^-,t,\tau,\grid,E]{\Est{*}{2}{#1}}
\newcommand{\E}[1]{\ensuremath{\mathcal{E}_{#1}}}

\newcommand{\grids}{\ensuremath{\mathbb{G}}\xspace}
\newcommand{\grid}{\mathcal{G}}
\newcommand{\gridn}[1][n]{\grid_{#1}}

\newcommand{\gridinit}{\ensuremath{\grid_\textrm{init}}}
\newcommand{\gridin}{\grid_{\textrm{in}}}

\newcommand{\gridold}{\grid_{\textrm{old}}}

\newcommand{\helm}[1][\ell]{\ensuremath{h_{\elm}}}

\newcommand{\ie}{\hbox{i.\,e.},\xspace}

\newcommand{\I}{\ensuremath{\mathbb{I}}}

\newcommand{\jump}[1]{\left[\negthinspace\left[{#1}\right]\negthinspace\right]}

\renewcommand{\L}{\ensuremath{\mathcal{L}}}

\newcommand{\marked}{\mathcal{M}}

\newcommand{\MARKREFINE}[1]{\texttt{MARK\_REFINE}\ensuremath{{#1}}\xspace}

\newcommand{\N}{\ensuremath{\mathbb{N}}}

\newcommand{\normal}{\ensuremath{\vec{n}}}

\newcommand{\norm}[2][\Omega]{\ensuremath{\|#2\|_{#1}}}

\DeclareMathOperator{\osc}{osc}
\newcommand{\oscG}[1][\grid]{\osc_{#1}}

\providecommand{\Poincare}{{Poincar{\'e}}\xspace}

\newcommand{\Ps}[1][s]{\ensuremath{\mathbb{P}_{#1}}}
\renewcommand{\paragraph}[1]{\noindent\raisebox{0pt}[10pt][0pt]{\textbf{#1.}}}

\newcommand{\Pron}[1][n]{\ensuremath{\Pi_{#1}}}
\newcommand{\ProG}[1][\grid]{\ensuremath{\Pi_{#1}}}

\newcommand{\R}{\ensuremath{\mathbb{R}}}
\DeclareMathOperator{\Res}{Res}
\newcommand{\Resh}{\ensuremath{\Res_h}}
\newcommand{\Rest}{\ensuremath{\Res_\tau}}

\newcommand{\REFINE}{\texttt{REFINE}}

\newcommand{\scp}[3][\Omega]{\ensuremath{\left\langle #2,\,#3\right\rangle}_{#1}}
\newcommand{\Scp}[3][\Omega]{\ensuremath{\langle #2,\,#3\rangle}_{#1}}

\newcommand{\side}{\ensuremath{S}\xspace}

\newcommand{\step}[1]{\noindent\raisebox{1.5pt}[10pt][0pt]{\tiny\framebox{$#1$}}\xspace}

\newcommand{\SOLVE}[1]{\texttt{SOLVE}\ensuremath{{#1}}}

\newcommand{\TAFEM}{\text{TAFEM}\xspace}
\newcommand{\TOLFIND}{\texttt{TOLFIND}}

\newcommand{\tn}[1][n]{\ensuremath{t_{#1}}}
\newcommand{\tauin}{\tau^{\textrm{in}}}
\newcommand{\TOL}{\ensuremath{\text{\texttt{TOL}}}}
\newcommand{\tol}{\ensuremath{\text{\texttt{tol}}}}

\newcommand{\TOLinit}{\ensuremath{\text{\texttt{TOL}}_0}}
\newcommand{\TOLcons}{\ensuremath{\text{\texttt{TOL}}_f}}
\newcommand{\tolcons}{\ensuremath{\text{\texttt{tol}}_f}}
\newcommand{\TOLst}{\ensuremath{\text{\texttt{TOL}}_{\grid\tau}}}
\newcommand{\tolst}{\ensuremath{\text{\texttt{tol}}_{\grid\tau}}}

\newcommand{\STADAPTATION}[1][(U,\tau,\bar f,\grid)]{\texttt{ST\_ADAPTATION}\ensuremath{{#1}}}

\newcommand{\UIn}[1][|I_n]{\ensuremath{U_{#1}}}
\newcommand{\Un}[1][n]{\ensuremath{U_{#1}}}
\newcommand{\U}{\ensuremath{\mathbb{U}}}
\newcommand{\hU}{\ensuremath{\mathcal{U}}}

\renewcommand{\vec}[1]{\ensuremath{\boldsymbol{#1}}}

\newcommand{\V}{\ensuremath{\mathbb{V}}}
\newcommand{\Vn}[1][n]{\ensuremath{\mathbb{V}_{#1}}}

\newcommand{\VG}[1][\grid]{\V(#1)}

\newcommand{\W}[1][0,T]{\ensuremath{\mathbb{W}{(#1)}}}

\newcommand{\Wnorm}[2][0,T]{\|{#2}\|_{\W[#1]}}

\algsetup{indent=2em}

\usepackage{subcaption}
\begin{document}


\title[A convergent adaptive dG(s) fem]{
  A convergent time-space adaptive dG(s) finite element method for parabolic problems motivated by equal error distribution}

\author[F.\,D. Gaspoz]{Fernando D.\ Gaspoz}
\address{Fernando D.\ Gaspoz, Institut f\"ur Angewandte Analysis und
  Numerische Simulation, Fachbereich Mathematik,
  Universit\"at Stuttgart,
  Pfaffenwaldring 57, D-70569 Stutt\-gart, Germany }
  \urladdr{www.ians.uni-stuttgart.de/nmh/}
\email{fernando.gaspoz@ians.uni-stuttgart.de}

\author[Ch.~Kreuzer]{Christian Kreuzer}
\address{Christian Kreuzer,
 Fakult\"at f\"ur Mathematik,
 Ruhr-Universit\"at Bochum,
 Universit\"atstrasse 150, D-44801 Bochum, Germany
 }%
\urladdr{http://www.ruhr-uni-bochum.de/ffm/Lehrstuehle/Kreuzer/index.html}
\email{christan.kreuzer@rub.de}

\author[K.\,G.~Siebert]{Kunibert~G.~Siebert}
\address{Kunibert~G.~Siebert, Institut f\"ur Angewandte Analysis und
  Numerische Simulation, Fachbereich Mathematik,
  Universit\"at Stuttgart,
  Pfaffenwaldring 57, D-70569 Stutt\-gart, Germany }
\urladdr{www.ians.uni-stuttgart.de/nmh/}
\email{kg.siebert@ians.uni-stuttgart.de}

\author[D.\,A. Ziegler]{Daniel A.\ Ziegler}
\address{Daniel A.\ Ziegler, Institut f\"ur Angewandte und
  Numerische Mathematik, Karlsruher Institut f\"ur Technologie,
  Englerstrsse 2, D-76131 Karlsruhe, Germany }
\urladdr{http://www.math.kit.edu/~daniel.ziegler}
\email{daniel.ziegler@kit.edu}



\begin{abstract}
  We shall develop a fully discrete space-time adaptive method for linear
  parabolic problems based 
  on new reliable and efficient a posteriori analysis for
  higher order dG(s) finite element discretisations. 
  The adaptive strategy is motivated by the principle of equally distributing the a
  posteriori indicators in time and the convergence of the method is
  guaranteed by the uniform energy estimate from
  \cite{KrMoScSi:12} under minimal assumptions on the 
  regularity  of the data. 
\end{abstract}

\date{\small\today}

\maketitle

\section{Introduction\label{sec:introduction}}

Let $\Omega$ be a bounded polyhedral domain in $\R^d$, $d\in\N$. 
We consider the linear parabolic partial differential equation
\begin{equation}
  \label{eq:strong}
  \begin{aligned}
    \Dt u + \L u &= f \qquad
    &&\text{in } \Omega\times (0,T) \\
    u &= 0 &&\text{on } \partial\Omega\times (0,T)\\
    u(\cdot,0) &= u_0 &&\text{in } \Omega.
  \end{aligned}
\end{equation}
Hereafter, $\L u=-\divo{\Am\nabla u} + cu$ is a second order
elliptic operator with respect to space and $\Dt u=\frac{\partial
  u}{\partial t}$ denotes the partial derivative with respect to
time. In the simplest setting $\mathcal{L}=-\Delta$, whence 
\eqref{eq:strong} is the heat equation. Precise
assumptions on data are provided in Section~\ref{ss:weak_formulation}.

The objective of this paper is the design and a detailed convergence
analysis of an efficient adaptive finite element method for solving 
\eqref{eq:strong} numerically. To this end, we resort 
to adaptive finite elements in space combined with a discontinuous
Galerkin \dGs time-stepping scheme in Section~\ref{ss:fem}. The conforming
finite element spaces are continuous piecewise polynomials of fixed degree over a
simplicial triangulation of the domain $\Omega$. 
In each single time-step, we reduce or enlarge the local time-step size
and refine and coarsen the underlying triangulation.  

The adaptive decisions are based on a posteriori
error indicators. 
{Numerous such estimators for various error notions are available in 
the literature. Error bounds in $L^\infty(L^2)$ can e.g. be found in
\cite{ErikssonJohnson:91} or \cite{LakkisMakridakis:06}, where the
latter result is based on the elliptic reconstruction technique, which was introduced in  
\cite{MakridakisNochetto:03} in the semi-discrete context.
The $L^2(H^1)$ respectively $L^2(H^1)\cap H^1(H^{-1})$ error bounds
in~\cite{Picasso:98,Verfurth:03} 
are based on energy techniques and
have  been used with a \dGs[0] time-stepping scheme in the adaptive 
methods and convergence analysis presented in \cite{ChenFeng:04,KrMoScSi:12}. }
For our purpose, we generalise the residual based 
estimator \cite{Verfurth:03} to higher order \dGs[s] schemes in Section~\ref{s:errest}. 
The estimator is build from five indicators:
an indicator for the initial error, indicators for the temporal and
spatial errors, a coarsening error indicator, and an indicator
controlling the so-called consistency error. 
It is important to notice that besides the first indicator all other indicators
accumulate in $L^2$ in time. 
The adaptation of the time-step-size uses informations of the indicators
for the temporal error and the consistency error. The adaptation of the
spatial triangulation is based on refinement by bisection using
information from the indicators for the spatial error and for the coarsening
error. Very recently an independently developed guaranteed a
posteriori estimator for higher order
$\dGs$ schemes is provided in~\cite{ErnSmearsVohralik:16} using
equilibrated flux based bounds for the spatial error. 

By now the convergence and optimality of adaptive methods for
stationary inf-sup stable respectively coercive
problems is well-estab\-lished
\cite{BiDaDe:04,CaKrNoSi:08,ChenFeng:04,DiKrSt:16,
Dorfler:96,DieningKreuzer:08,KreuzerSiebert:11,
MoSiVe:08,MekchayNochetto:05,MoNoSi:00,MoNoSi:02,
Siebert:11,Stevenson:07};   
compare also with the overview article \cite{NoSiVe:09}.
The essential design principle motivating the adaptive strategies in 
most of the above methods is the equal distribution of the error. 
The importance of this principle is highlighted by the near 
characterisations of nonlinear approximation classes with the help 
of a thresholding algorithm in \cite{BiDaDePe:02,GaspozMorin:14}.

In contrast to the situation for above mentioned problems, the
convergence analysis of adaptive approximation of
time-dependent problems is still in its infancy. In 
\cite{SchwabStevenson:09} optimal computational complexity of
an adaptive wavelet method for parabolic problems is proved using a
symmetric and coercive discretisation based on a least squares formulation. 
To our best
knowledge, there exist only two results \cite{ChenFeng:04,KrMoScSi:12}
concerned with a rigorous convergence analysis of time-space adaptive
finite element methods. In \cite{ChenFeng:04}, it is
proved for the \dGs[0] time-stepping scheme, that each single
time-step terminates and that the error of the computed approximation
is below a prescribed tolerance when the final time is
reached. This, however, is not guaranteed and thus theoretically the adaptively generated
sequence of time instances $\{t_n\}_{n\ge 0}$ may not be finite and
such that $t_n\to t_\star<T$ as $n\to\infty$. This drawback has been overcome
in~\cite{KrMoScSi:12}  with the help of an a priori computed minimal
time-step size 
in terms of the
right-hand side $f$ and the discrete initial value $U_0$. 
However, neither design of the two methods heeds the principle of
equally distributing the error. Let us shed some light on this fact
with the help of the initial value problem 
\begin{align*}
  \Dt u + u = f\quad\text{in}~(0,T) \qquad\text{and}\qquad u(0)=u_0.
\end{align*}
Let $0=t_0<t_1<\ldots<t_N=T$ be some partition of $(0,T)$. Using the 
\dGs[0] time-stepping scheme we obtain $\{U_n\}_{n=0}^N$, such that 
\begin{align*}
  \frac{U_n-U_{n-1}}{\tau_n}+U_n=f_n:=\frac1{\tau_n}\int_{t_{n-1}}^{t_n}f\dt,\quad
  n=1,\ldots, N, \qquad U_0=u_0,
\end{align*}
where $\tau_n=t_n-t_{n-1}$. Let $\hU$ be the piecewise affine
interpolation $\hU$ of the nodal values $\{U_n\}_{n=0}^N$. Then we
have with
Young's inequality, that 
\begin{align*}
  \int_0^T\frac12 \Dt|u-\hU|^2+|u-\hU|^2\dt&=
                                                               \sum_{n=1}^N\int_{t_{n-1}}^{t_n}
                                                               (f-f_n)(u-\hU)+(U_n-\mathfrak{
                                                               u})(
                                                               u-\hU)\dt
  \\
  &\le \sum_{n=1}^N\int_{t_{n-1}}^{t_n} |f-f_n|^2 +
    |U_n-\hU|^2+ \frac12 |u-\hU|^2\dt.  
\end{align*}
A simple computation reveals
$\int_{t_{n-1}}^{t_n}|U_n-\hU|^2\dt= \frac13 \tau_n
|U_n-U_{n-1}|^2$. This term and $\int_{t_{n-1}}^{t_n}|f-f_n|^2\dt$
are the so-called time and consistency a posteriori indicators. 
In order to illustrate the basic differences in the design of the
adaptive schemes, we shall concentrate 
on the time indicator.
In~\cite{ChenFeng:04,KrMoScSi:12} the partition is constructed such that 
\begin{align*}
  |U_n-U_{n-1}|^2 \le \frac{\TOL^2}{T},\quad\text{which implies}
  \quad \sum_{n=1}^N \tau_n |U_n-U_{n-1}|^2\le
  \sum_{n=1}^N \tau_n \frac{\TOL^2}{T}=\TOL^2,
\end{align*}
i.e. the accumulated indicator is below the
prescribed tolerance $\TOL$.
We call this the $L^\infty$-strategy and remark that it does not aim  
at equally distributing the local
indicators. 
In contrast to this, we shall use the $L^2$-strategy 
\begin{align*}
\tau_n|U_n-U_{n-1}|^2\le\tol^2.
\end{align*}
Thanks to the  uniform energy bound 
\begin{align}\label{eq:energy-ode}
  \sum_{n=1}^N|U_n-U_{n-1}|^2\le \int_0^T|f|^2\dt +|U_0|^2
\end{align}
(see Corollary~\ref{cor:uniform_bound} below) we conclude then, that 
\begin{multline*}
  \sum_{n=1}^N \tau_n |U_n-U_{n-1}|^2=\sum_{\tau_n\le\delta}^N \tau_n
  |U_n-U_{n-1}|^2+\sum_{\tau_n>\delta}^N \tau_n |U_n-U_{n-1}|^2
  \\
  \le \delta\, \Big(\int_0^T|f|^2\dt +|U_0|^2\Big)+\frac{T}{\delta}
    \tol^2
    =\left(\frac{T}{\int_0^T|f|^2\dt +|U_0|^2}\right)^{\frac12}\,\tol
\end{multline*}
where $\delta=(T/(\int_0^T|f|^2\dt +|U_0|^2))^{1/2}$. Taking
$\tol=\TOL^2/\delta$ guarantees that the error is below the
prescribed tolerance $\TOL$. 

These arguments directly generalise to 
semi-discretisations of~\eqref{eq:strong} in time. In the case of a full
space-time discretisation of~\eqref{eq:strong} additional
indicators are involved, for which a control similar
to~\eqref{eq:energy-ode} is not available. 
We therefore enforce that these indicators are bounded by the 
time or the consistency indicator. If these indicators are
equally distributed in time, then this results also  in an
equal distribution of the other indicators. Otherwise, we shall use the
$L^\infty$-strategy from \cite{ChenFeng:04,KrMoScSi:12} as a backup strategy. 
The detailed algorithm \TAFEM for \eqref{eq:strong} is
presented in Section~\ref{sec:tafem} and its
convergence analysis is given in Section~\ref{sec:convergence}.

The
advantage of our new approach over the algorithms in
\cite{ChenFeng:04,KrMoScSi:12} is twofold. First, 
from the fact that the \TAFEM aims in an equal distribution of the
error, 
we expect an improved performance.
Second, we use an $L^2$-strategy for the consistency error, which 
requires only $L^2$-regularity of $f$ in time in stead of the
$H^1$-regularity needed for the $L^\infty$-strategy in
\cite{ChenFeng:04,KrMoScSi:12}. 
This makes the proposed method suitable for problems, where
the existing approaches may even fail completely.
We conclude the paper in Section~\ref{sec:numerics} by comments on the
implementation in \DUNE~\cite{DUNE:16} and some numerical experiments. The experiments 
confirm the expectations and show a more than
competitive performance of our algorithm~\TAFEM.

\section{The Continuous and Discrete Problems\label{s:prob+disc}}

In this section, we state the weak formulation of the continuous problem together with the assumptions on
data. Then the discretisation by adaptive finite
elements in space combined with the \dGs scheme in time is introduced.
\subsection{The Weak Formulation\label{ss:weak_formulation}}
For $d\in\N$, let $\Omega\subset\R^d$ be a bounded, polyhedral domain
that is meshed by some conforming simplicial mesh $\gridinit$. 
We denote by $H^1(\Omega)$ the Sobolev space of square integrable 
functions  $L^2(\Omega)$ whose first derivatives are in
$L^2(\Omega)$ and we let $\V:=H_0^1(\Omega)$ be the space of functions in
$H^1(\Omega)$ with vanishing trace on $\partial\Omega$. For 
any measurable set $\omega$ and $k\in\N$,  we denote by $\norm[\omega]{\cdot}$ the
$L^2(\omega;\R^k)$ norm, whence 
\begin{math}
  \norm[H^1(\Omega)]{v}^2=\norm[\Omega]{v}^2 + \norm[\Omega]{\nabla v}^2.
\end{math}

We suppose that the data of \eqref{eq:strong} has the following properties:
$\Am:\Omega\rightarrow \R^{d\times d}$ is piecewise Lipschitz over
$\gridinit$ and is symmetric positive definite with eigenvalues $0<
a_*\leq a^*<\infty$, \ie
\begin{equation}\label{A-bounds}
  a_*|\boldsymbol{\xi}|^2\leq\Am(x)\boldsymbol{\xi}\cdot\boldsymbol{\xi}\leq
  a^*|\boldsymbol{\xi}|^2,
  \qquad \text{for all } \boldsymbol{\xi}\in\R^d,\; x\in\Omega,
\end{equation}
$c\in L^\infty(\Omega)$ is non-negative, \ie $c\geq0$ in $\Omega$,
$f\in L^2((0,T);L^2(\Omega))= L^2(\Omega\times(0,T))$, and $u_0\in L^2(\Omega)$.  

We next turn to the weak formulation of \eqref{eq:strong}; compare with
\cite[Chap.~7]{Evans:10}. We let $\B:\V\times
\V\rightarrow\R$ be
the symmetric bilinear form associated to the weak form
of elliptic operator $\mathcal{L}$, \ie
\begin{align*}
  \bilin{w}{v} \definedas \int_\Omega \Am\nabla v \cdot \nabla w 
  +c\,vw \,dV \qquad\text{for all } v,\,w\in\V.
\end{align*}
Recalling the Poincar\'e-Friedrichs inequality 
\begin{math}
  \norm[\Omega]{v}\le C(d,\Omega)\norm[\Omega]{\nabla v}
\end{math}
for all $v\in \V$ \cite[p.~158]{GilbargTrudinger:01} we deduce from
\eqref{A-bounds} that $\B$ is a scalar product on $\V$ with
induced norm
\begin{align*}
  \Enorm{v}^2:=\bilin{v}{v}= \int_\Omega \Am \nabla v\cdot\nabla v +
cv^2\,dV  \qquad \text{for all } v\in H_0^1(\Omega).
\end{align*}
This \emph{energy norm} is equivalent to the
$H^1$-norm $\norm[H^1(\Omega)]{\cdot}$ and we shall use the energy
norm in the subsequent analysis. We denote the 
restriction of the energy norm to some subset $\omega\subset\Omega$
by $\Enorm[\omega]{\cdot}$ and let $\V^*\definedas H^{-1}(\Omega)$
be the dual space of $H^1_0(\Omega)$ equipped with the operator norm
\begin{math}
\Enorm[*]{g}\definedas \sup_{v\in\V} \frac{\dual{g}{v}}{\Enorm{v}}.
\end{math}

The weak solution space
\begin{align*}
  \W\definedas \left\lbrace u\in L^2(0,T; \V) \mid
    \Dt u\in L^2(0,T;\V^*)
\right\rbrace.
\end{align*}
is a Banach space endowed with the norm
\begin{align*}
  \Wnorm{v}^2=\int_{0}^{T}\Enorm[*]{\Dt v}^2 +
  \Enorm{v}^2\dt+\norm{v(T)}^2,\qquad v\in\W[0,T].
\end{align*}
Moreover, it is continuously embedded into $C^0([0,T];L^2(\Omega))$; see e.g. \cite[Chap.~5]{Evans:10}.

After these preparations, we are in the position to state the weak
formulation of \eqref{eq:strong}: A function $u\in\W$ is a weak
solution to \eqref{eq:strong} if it satisfies
\begin{subequations}\label{eq:weak}
  \begin{alignat}{2}\label{eq:weak.a}
    \dual{\Dt u(t)}{v} + \bilin{u(t)}{ v} &= \scp{f(t)}{v}
    \quad&&\text{for all } v\in \V,\; \text{a.e. } t \in(0,T),\\
    \label{eq:weak.b}
    u(0) &= u_0.
  \end{alignat}
\end{subequations}
Hereafter, $\scp{\cdot}{\!\cdot}$ denotes the $L^2(\Omega)$ scalar
product. Since the operator $\mathcal{L}$ is elliptic, problem
\eqref{eq:weak} admits for any $f\in L^2(0,T;L^2(\Omega))$ and $u_0\in
L^2(\Omega)$ a unique weak solution; compare e.g. with \cite[Chap.~7]{Evans:10}.

\subsection{The Discrete Problem\label{ss:fem}}

For the discretization of \eqref{eq:weak} we use adaptive finite elements in
space and a \dGs scheme with adaptive time-step-size controle. 

\paragraph{Adaptive Grids and Time Steps} For the adaptive space
discretization we restrict ourselves to simplicial grids and
local refinement by bisection; compare with
\cite{Bansch:91,Kossaczky:94,Maubach:95,Traxler:97} as well as
\cite{NoSiVe:09,SchmidtSiebert:05} and the references therein.  To be
more precise, refinement is based on the initial conforming
triangulation $\gridinit$ of $\Omega$ and a procedure $\REFINE$ with the
following properties. Given a conforming triangulation $\grid$ and a
subset $\marked\subset\grid$ of \emph{marked elements}. Then 
\begin{displaymath}
  \REFINE(\grid,\marked)
\end{displaymath}
outputs a conforming refinement $\grid_+$ of $\grid$ such that all elements in
$\marked$ are bisected at least once. In general, additional elements are refined in
order to ensure conformity. The input $\grid$ can either be
$\gridinit$ or the output of a previous application of $\REFINE$.  The
class of all conforming triangulations that can be produced from
$\gridinit$ by finite many applications of $\REFINE$, we denote by $\grids$.  For $\grid\in\grids$
we call $\grid_+\in\grids$ a \emph{refinement} of $\grid$ if $\grid_+$ is
produced from $\grid$ by a finite number of applications of $\REFINE$
and we denote this by $\grid\leq\grid_+$ or
$\grid_+\ge\grid$. Conversely, we call
any $\grid_-\in\grids$ with $\grid_-\le \grid$ a
\emph{coarsening} of $\grid$.  

Throughout the discussion we only deal with conforming grids, this
means, whenever we refer to some triangulations $\grid$, $\grid_+$, and $\grid_-$ we
tacitly assume $\grid,\grid_+,\grid_-\in\grids$.
One key property of the refinement by bisection is uniform shape
regularity for any $\grid\in\grids$. This means that all constants depending
on the shape regularity are uniformly bounded depending on
$\gridinit$.

For the discretization in time we let $0=t_0<t_1<\dots<t_N=T$ be a
partition of $(0,T)$ into half open subintervals $I_n=(t_{n-1},t_n]$ with
corresponding local time-step sizes $\tau_n:=\abs{I_n} = t_n-t_{n-1}$, $n=1,\dots,N$.

\paragraph{Space-Time Discretization}
For the spatial discretization we use Lagrange finite elements. This is, 
for any $\grid\in\grids$ the finite element space $\V(\grid)$ consists of
all continuous, piecewise polynomials of fixed degree $\ell\ge1$ over
$\grid$ that vanish on $\partial\Omega$. This gives a 
conforming discretization of $\V$, \ie $\V(\grid)\subset\V$.
Moreover, Lagrange finite elements give nested spaces, \ie
$\V(\grid)\subset\V(\grid_+)$ whenever $\grid\le\grid_+$. 

We denote by $\gridn[0]$ the triangulation at $t_0=0$ and for $n\ge
1$, we denote by $\gridn$ the grid in $I_n$ and let
$\V_n=\VG[\gridn]$, $n=0,\ldots, N$, be 
the corresponding finite element spaces. For
$\grid\in\grids$ we denote by $\ProG\colon L^2(\Omega)\to\V(\grid)$
the $L^2$ projection onto $\V(\grid)$ and set
$\ProG[n]\definedas\ProG[\gridn]$. 

On each time interval, the discrete approximation is polynomial in time 
over the corresponding spatial finite element
space. Let $s\in\N_0$, for any real vector space $\U$ and interval
$I\subset \R$, we denote by
\begin{align*}
  \Ps\big(I,\U\big):=\Big\{t\mapsto \sum_{i=0}^s t^iV_i:V_i\in\U\Big\}
\end{align*}
the space of all polynomials with degree less or equal $s$ over
$\U$. We write $\Ps(\U):=\Ps(\R,\U)$ and $\Ps:=\Ps(\R)$.

Furthermore, for an interval $I\subset (0,T)$ we let
\begin{displaymath}
  f_I\in \Ps(I,L^2(\Omega))
\end{displaymath}
be the best-approximation of $f_{|I}$ in $L^2(I,L^2(\Omega))$. 
In particular we use
$f_n\definedas f_{I_n}$  as a time-discretization of $f$ on $I_n$. 
For $s=0$, $f_I=\fint_I f\dt$ is the mean value of $f$ on $I$.

In the following, we introduce the so called discontinuous Galerkin
time-stepping scheme $\dGs$ of degree $s$, where $\dGs[0]$
is the well know implicit Euler scheme.
To this end, we denote for $n\ge 1$ the actual grid on $I_n$ by $\gridn$ and
let $\Vn=\V(\gridn)$ be the corresponding finite element space.
We start with a suitable initial refinement
$\gridn[0]$ of $\gridinit$ and an approximation 
$\Un[0]=\Pi_0 u_0=\Pi_{\gridn[0]} u_0\in\Vn[0]$ of the initial value
$u_0$. Note that in principle, any suitable interpolation operator can
be used instead of $\Pi_0$.
We then inductively compute for $n>0$ a
solution $\UIn\in\Ps(\Vn)$ to the problem
\begin{align}\label{eq:discrete}
  \begin{split}
    \int_{I_n}\scp{\Dt \UIn}{V} + \bilin{\UIn}{V}\dt
    &+\scp{\jump{U}_{n-1}}{V(t_{n-1})}
   = \int_{I_n}\scp{f_{n}}{ V} \dt
  \end{split}
\end{align}
for  all  $V\in\Ps(\Vn)$. Thereby $f_n\definedas f_{I_n}$ and $\jump{U}_{n-1}$ denotes the jump
\begin{align*}
  \jump{U}_{n-1}:=U_{n-1}^+-U_{n-1}^-,
\end{align*}
of $U$ across $t_{n-1}$, 
where we used $ U_{n-1}^+:=\lim_{t\downarrow
    t_{n-1}}\UIn(t)$,  $U_{n}^-:=\UIn(t_{n})$,
  $n=1,\ldots,N$, and 
  $U_0^-:=U_0$.
Note that with this definition we have 
$U_{n-1}^-=U(t_{n-1})$. The solution $U$ is uniquely defined
\cite{Thomee:06} 
and we will see below that \eqref{eq:discrete} 
is equivalent to an $s+1$ dimensional second
order elliptic system. Note that $U$ is
allowed to be discontinuous across the nodal points $t_0,\ldots,t_N$ and
hence in general $U\not\in\W$. 

In order to construct from $U$ a conforming function, 
we recall that the \dGs schemes are
closely related to Runge Kutta RadauIIA collocation methods; see
e.g. \cite{AkMaNo:09}. The corresponding RadauIIA quadrature
formula with abscissae $c_1, \ldots, c_{s+1}$ and weights
$b_1,\ldots,b_{s+1}$ is exact of degree $2s$. In fact, we have
\begin{align}
  \label{eq:RadauIIA}
  \sum_{j=1}^{s+1}b_j P(c_j) =\int_0^1P(t)\dt\qquad \text{for all}~P\in\Ps[2s].
\end{align}

We define $\hU\in \W$, $\hU_{|I_n}\in\Ps[s+1](\V)$ as the piecewise interpolation
of $U$ at the local RadauIIA points $t_n^j:=t_{n-1}+c_j\tau_n$, \ie
\begin{subequations}\label{eq:hU}
  \begin{align}\label{eq:hUa}
    \hU(t_{n}^j) &= \UIn(t_{n}^j)\in\Vn, \qquad j =1,\ldots,s+1.
  \intertext{The continuous embedding of  $\W$ in $C^0([0,T];L^2(\Omega))$
  additionally enforces}
\label{eq:hUb}
    \hU(t_{n-1}) &= \Un[n-1]^-\in\Vn[n-1].
  \end{align}
\end{subequations}
Hence $\hU$ is uniquely defined by 
\begin{align}\label{eq:UhLagrange}
  \hU_{|I_n}:= \sum_{j=0}^{s+1} L_j\big(\tfrac{t-t_{n-1}}{\tau_n}\big)\,
  U(t_n^j); 
\end{align}
with the Lagrange polynomials
\begin{align}\label{eq:Lagrange}
  L_j(t):= \prod_{\atop{i=0}{i\neq j}}^{s+1}
  \frac{t-c_j}{c_i-c_j}\in\Ps[s+1],\qquad j=0,\ldots,s+1
\end{align}
and $c_0:=0$. 
Using integration by parts with
respect to time, \eqref{eq:RadauIIA}, and  
\eqref{eq:hU}, we observe  that \eqref{eq:discrete} is equivalent to  
\begin{align}\label{eq:discreteb}
   \int_{I_n}\scp{\Dt \hU}{V} + \bilin{U}{V}\dt
    &
   = \int_{I_n}\scp{f_{n}}{ V} \dt
\end{align}
for all $n=1,\ldots,n$ and $V\in\Ps(\Vn)$.

We emphasize that $\hU(t)$ is a finite element function,
since for $t\in I_n$, we have $\hU(t)\in\V(\gridn[n-1]\oplus\gridn)\subset\V$, 
where $\gridn[n-1]\oplus\gridn$
is the smallest common refinement of $\gridn[n-1]$ and $\gridn$, which we
call \emph{overlay}.
Continuity of $\hU$ in time, in combination with $\hU(t)\in\V$ for all
$t\in I$ then implies $\hU\in\W$. 
\begin{rem}
  For $s=0$ we see from~\eqref{eq:discrete} that in each time-step
  $n\in\N$, we need to solve for partial differential operators of the form $-\Delta+\mu$ with
  $\mu=\frac1{\tau_n}$ in order to
  compute $U_n$. Unfortunately, for $s>0$, though still coercive,
  \eqref{eq:discrete} becomes a $s+1$ dimensional coupled non-symmetric
  system. Recently,  in~\cite{Smears:15} a PCG method for a
  symmetrisation of~\eqref{eq:discrete} is proposed, which is fully
  robust with respect to the discretisation parameters 
  $s$ and $\tau$, provided a solver for the
  operator $-\Delta+\mu$, $\mu\ge0$ is available. 
\end{rem}

\section{A Posteriori Error Estimation\label{s:errest}}
One basic ingredient of adaptive methods are a posteriori error
indicators building up a reliable upper bound for the error in terms
of the discrete solution and given data. The \dGs[0] method
corresponds to the implicit Euler scheme and residual based
estimators for the heat equation can be found in
\cite{Verfurth:03}. In this section we generalize this result 
and prove reliable and efficient residual based estimators for
\dGs schemes  \eqref{eq:discrete}, with arbitrary $s\in\N_0$.

Some arguments in this section are  straight forward
generalizations of those
in \cite{Verfurth:03} and we only sketch their proofs, others are
based on new ideas and therefore we shall prove them in detail.

\subsection{Equivalence of Error and Residual\label{s:err=res}}
In order to prove residual based error estimators, one first has to
relate the error to the residual. To this end we 
note that \eqref{eq:weak} can be taken as an operator equation in
$L^2(0,T;\V^*)\times L^2(\Omega)$. Its residual $\Res(\hU)$ in $\hU\in\W$ is given by  
\begin{align}\label{eq:Res}
  \begin{split}
    \scp[]{\Res(\hU)}{v}&=\scp[]{\Dt(u-\hU)}{v}+\bilin{u-\hU}{v}
    \\
    &=\scp{f-\Dt \hU}{v}-\bilin{\hU}{v}\qquad\qquad\text{for all
      $v\in\V$.}
  \end{split}
\end{align}

From \cite{TantardiniVeeser:16b}, we have the following identity
between the residual and the error.
\begin{prop}[Abstract Error Bound]\label{p:err=res}
  Let $u\in\W$ be the solution of \eqref{eq:weak} and let 
  $\hU\in\W$ be the approximation defined in \eqref{eq:hU} for
  time instances $t_0=0,\ldots,t_N=T$ and time-step sizes
  $\tau_n:=t_{n}-t_{n-1}$, $n=1,\ldots,N$. Then it holds for $0\le k\le N$, that 
  \begin{subequations}\label{eq:err=res}
    \begin{align}\label{eq:err<res}
      \norm[\W]{u-\hU}^2&= \norm{u_0-U_0}^2 + 
      \norm[{L^2(0,T;\V^*)}]{\Res(\hU)}^2
      \intertext{and}\label{eq:res<err}
      \norm[L^2(I_n,\V^*)]{\Res(\hU)}^2&\le 2\big\{
      \norm[L^2(I_n;\V^*)]{\Dt(u-\hU)}^2+\norm[L^2(I_n;\V)]{u-\hU}^2\big\}.
    \end{align}
  \end{subequations}

\end{prop}

The rest of
this section concentrates on proving computable upper and lower bounds for
the error. We note that the initial error $\norm{u_0-U_0}$
in \eqref{eq:err=res} is already a posteriori computable, whence 
it remains to estimate the dual norm of the residual. However, there is
another issue of separating the influence of the temporal and the spatial
discretization to the error. 
In particular, defining the
temporal residual $\Rest(\hU)  \in L^2(0,T;\V^*)$ as
\begin{align}
  \label{df:rest}
  \scp[]{\Rest(\hU)}{v}&\definedas\bilin{U-\hU}{v}
  \intertext{and the spatial residual $\Resh(\hU)\in L^2(0,T;\V^*)$
    as}
  \label{df:resh}
  \scp[]{\Resh(\hU)}{v}&\definedas\scp[]{f_n-\Dt\hU}{v}-\bilin{U}{v}
  \qquad\text{on}\quad I_n,
\end{align}
we obtain
\begin{align}\label{eq:resdeco}
  \Res(\hU)= \Rest(\hU)+\Resh(\hU)+f-f_n \qquad\text{on}\quad I_n.
\end{align}
In what follows we use this decomposition to prove separated time and
space error indicators, which build up a reliable and efficient bound
for the error. 

\subsection{Temporal Residual\label{s:rest}}
Recalling the definition of the Lagrange polynomials
\eqref{eq:Lagrange}, we have the local unique representation 
\begin{align*}
  \UIn(t)= U_{n-1}^+L_0(\tfrac{t-t_{n-1}}{\tau_n})+\sum_{j=1}^{s+1}U(t_{n}^j)
    \,L_j(\tfrac{t-t_{n-1}}{\tau_n})
    \in\Ps(\V_n)
\end{align*}
for all $t\in I_n$. Hence, by \eqref{eq:UhLagrange} we get
\begin{align*}
   \hU(t)-U(t)&=(U_{n-1}^--U_{n-1}^+)\,L_0(\tfrac{t-t_{n-1}}{\tau_n})
\end{align*}
and thanks to~\eqref{eq:hU} and \eqref{eq:RadauIIA}, we obtain
\begin{align}\label{eq:rest}
  \begin{split}
    \int_{I_n}\norm[\V^*]{\Rest(\hU)}^2\dt&=
    \int_{I_n} \Enorm{U- \hU}^2 \dt
    =\Enorm{ U_{n-1}^--
      U_{n-1}^+}^2\int_{I_n}|L_0(\tfrac{t-t_{n-1}}{\tau_n})|^2\dt
    \\
    &=\tau_n\, C_\tau\,\Enorm{U_{n-1}^-- U_{n-1}^+}^2,
  \end{split}
\end{align}
where
$C_\tau=C_\tau(s):=\int_{0}^1|L_0(t)|^2\dt$. 

\begin{rem}\label{r:Ct}
  Observing that the RadauIIA abscissae are the roots of the
  polynomial $\lambda_{s}(2t-1)-\lambda_{s+1}(2t-1)$ and 
  $\lambda_{s}(-1)=(-1)^{s}$, 
  with the Legendre polynomials $\lambda_n$, $n\in\N_0$, it follows
  that we have the representation 
  $$L_0(t)=\frac{(-1)^s}2(\lambda_{s}(2t-1)-\lambda_{s+1}(2t-1))$$ and it 
  can be easily shown that
  $C_\tau=\frac14(\frac{1}{2s+3}+\frac1{2s+1})$.
\end{rem}

\subsection{The Spatial Residual\label{s:Resh}}
In this section we estimate the spatial residual.

\begin{lem}\label{l:UpSpat}
  Let $U$ be the approximation of \eqref{eq:discrete} to the solution $u$
  of \eqref{eq:weak} and let $\hU$ be its interpolation defined by
  \eqref{eq:hU}. Then there exists a constant $C_\grids>0$, such that
  \begin{align*}
    \int_{I_n}\norm[\V^*]{\Resh(\hU)}^2\dt&\le C_\grids \sum_{\elm\in\grid_n}\int_{I_n}
  h_\elm^2\norm[\elm]{\Dt \hU+\L U-f_n}^2+h_\elm\norm[\partial E]{J(U)}^2\dt
  \end{align*}
  for all $1\le n\le N$. Thereby, for $V\in\Vn$ we denote by $J(V)|_S$ for an interior side $\side$
the jump of the normal flux $\Am\nabla V\cdot\normal$ across $\side$
and for boundary sides $\side$ we set $J(V)|_S\equiv0$.
The mesh-size of an element
$\elm\in\grid$ is given by $h_\elm\definedas |\elm|^{1/d}$.
\end{lem}

\begin{proof}
  Recalling \eqref{df:resh},
  we first observe that 
  $\norm[\V^*]{\Resh(\hU)}^2\in\Ps[2s]$,
  whence by \eqref{eq:RadauIIA} we have
  \begin{align*}
    \int_{I_n} \norm[\V^*]{\Resh(\hU)}^2\dt = \tau_n \sum_{j=1}b_j
    \norm[\V^*]{\Resh(\hU)(t_{n}^j)}^2. 
  \end{align*}
  Therefore, it suffices to estimate $\norm[\V^*]{\Resh(\hU)}^2$ 
  at the abscissae of the RadauIIA quadrature formula. For arbitrary
   $V_j\in\V_n$, $j=1,\ldots,s+1$ choose
   $V\in\Ps(\V_n)$ in \eqref{eq:discreteb} such that 
   $V(t+c_i\tau_n)=V_j\delta_{ij}$, $1\le i\le s+1$. 
   Then  exploiting again \eqref{eq:RadauIIA} yields the Galerkin orthogonality 
  \begin{align*}
    \scp[]{\Resh(\hU)(t_n^j)}{V_j}&= 0 \qquad j=1,\ldots,s+1.
  \end{align*}
  Since $V_j\in\Vn$ was arbitrary, we have for any $v\in\V$, that
  \begin{align*}
    \scp[]{\Resh(\hU)(t_n^j)}{v}&=\scp[]{\Resh(\hU)(t^j_n)}{v-V}\qquad \text{for
      all } V\in\V_n.
  \end{align*}
  Using integration by parts with respect to the space variable,
  the Cauchy-Schwarz inequality, the scaled trace inequality, and choosing $V$ as
  a suitable interpolation of $v$, we arrive at 
  \begin{align*}
    \norm[\V^*]{\Resh(\hU)(t^j_n)}^2 \le C_\grids \sum_{\elm\in\grid_n}\Big\{&
    h_\elm^2\norm[\elm]{(\Dt \hU+\L U-f_n) (t^j_n)}^2
   +h_\elm\norm[\partial E]{J(U) (t^j_n)}^2\Big\}.
  \end{align*}
  The right hand side is a pointwise  evaluation of a polynomial of
  degree $2s$ and thus the claimed upper bound follows from~\eqref{eq:RadauIIA}.
\end{proof}

The following result shows that the spatial indicators are locally
efficient as well. 
\begin{lem}\label{l:LowSpat}
  Under the conditions of Lemma \ref{l:UpSpat}, we have
  \begin{align*}
    \sum_{\elm\in\grid_n}\int_{I_n}
  h_\elm^2\norm[\elm]{\Dt \hU+\L U-f_n}^2&+h_\elm\norm[\partial
  E]{J(U)}^2\dt
  \\
  &\le C  \big\{   \int_{I_n}\norm[\V^*]{\Resh(\hU)}^2+\oscG[\grid_n]^2(f_n,\hU)\dt\Big\},
  \end{align*}
  where
  \begin{align*}
    \oscG[\grid_n]^2(f_n, \hU):=\sum_{\elm\in\grid_n}
    h_\elm^2\norm[\elm]{\Dt \hU+\L U-f_n-R_\elm }^2
  + h_\elm\norm[\partial E]{J(U)-J_E}^2
  \end{align*}
  with at time $t\in I_n$ pointwise $L^2(\Omega)$-best approximations $R_\elm(t)\in \Ps[2\ell-2](E)$
  respectively $J_E(t)_{|\side}\in \Ps[2\ell-1](\side)$ for each side
  $\side\subset \partial\elm$. The constant $C>0$ depends solely on
  the shape regularity of $\grids$. 
\end{lem}

\begin{proof}
 With the same arguments as in the proof of Lemma \ref{l:UpSpat},
  for each $1\le j\le s+1$ it suffices to
  prove that 
  \begin{align*}
    \begin{split}
       C_\grids\sum_{E\in\grid_n}h_\elm^2\norm[\elm]{(\Dt \hU+\L U-f_n)(t^j_n)}^2
       &+h_\elm\norm[\partial E]{J(U)(t^j_n)}^2
       \\
       &\leq
       C\big\{\norm[\V^*]{\Resh(\hU)(t^j_n)}^2+\oscG[\grid_n]^2(f_n,U)(t^j_n)\big\}
   \end{split}
  \end{align*}
  This however, follows with standard
  techniques used in a posteriori estimation of elliptic second order
  problems; see e.g. \cite{Verfuerth:2013,MekchayNochetto:05} and
  compare with the case of the implicit Euler scheme $s=0$ in \cite{Verfurth:03}.
\end{proof}

\subsection{Estimation of the Error\label{s:err_est}}
By means of the decomposition of the residual \eqref{eq:resdeco}, we can
combine the above results to obtain a reliable and
efficient error estimator for \eqref{eq:strong}. To this end, we
introduce the following error indicators for the sake of brevity of presentation: 
For $\grid\in\grids$ and $v\in\V$, the estimator for the initial value is
given by 
\begin{subequations}\label{eq:indicators}
\begin{align}
  \label{df:inest}
  \Einit:=\norm{v-\mathcal{I}_\grid v}^2
\end{align}
For $f\in
L^2(0,T;L^2(\Omega))$, $t_\star\in (0,T)$ and $I=(t_\star,t_\star+\tau]\subset (t_\star,T]$,
the so called consistency
error, which  is inherited by the decomposition of the residual
\eqref{eq:resdeco} is defined by
  \begin{align}
    \label{df:fest}
    \Econs[f,t_\star,\tau]&:=3\inf_{\bar f \in
      \Ps(L^2(\Omega))}\int_{I}\norm{f-\bar f}^2\dt.
    \intertext{For $v^-,v^+\in\V$,
      $\grid\in\grids$, $V\in\Ps(\VG)$, $\elm\in\grid$, and
      $g\in\Ps(L^2(\Omega))$ the indicator
      }
    \label{df:test}
    \Etc[v^+,v^-,\tau]&:=\tau\, 3\, C_{\tau}\, \Enorm{v^--
      v^+}^2
    \intertext{is motivated by
  \eqref{eq:rest} and Lemma \ref{l:UpSpat} suggests to define the
                        spatial indicators by}
    \label{df:sest}
    \begin{split}
      \Espace[V,v^-,t_\star,\tau,g,\grid,\elm]&
      :=3\,C_\grids\int_{I}
      h_\elm^2\norm[\elm]{\Dt \mathcal{V}+\L V-g}^2
      +h_\elm\norm[\partial E]{J(V)}^2\dt
      \\
      &=3\,C_\grids\,\tau\,\sum_{j=1}^{s+1}b_j\Big\{
      h_\elm^2\norm[\elm]{(\Dt \mathcal{V}+\L V-g)(t_\star+c_j\tau)}^2
      \\
      &\phantom{=3\,C_\grids\,\tau\,\sum_{j=1}^{s+1}b_j\Big\{ }+h_\elm\norm[\partial E]{J(V)(t_\star+c_j\tau)}^2\Big\}.
    \end{split}
  \end{align}
 \end{subequations}
 Here we have used, analogously to \eqref{eq:UhLagrange}, that 
  \begin{align}\label{eq:Vh}
    \mathcal{V}(t):= \sum_{j}^{s+1} L_j\big(\tfrac{t-t_\star}{\tau}\big)\,
  V(t_\star+c_j\tau)+L_0\big(\tfrac{t-t_\star}{\tau}\big) v^-\in\Ps[s+1](\V).
  \end{align}

\begin{prop}[Upper Bound]\label{p:upper}
Let $u\in\W$ be the solution of \eqref{eq:weak} and let 
  $\hU\in\W$ be the approximation defined in \eqref{eq:hU} for
  time instances $t_0=0,\ldots,t_N=T$ and time-step sizes
  $\tau_n:=t_{n}-t_{n-1}$, $n=1,\ldots,N$. Then we have the estimate
  \begin{align*}
    \norm[\W]{u-\hU}^2&\le \Einit[u_0,\grid_0]+\sum_{n=1}^N\Big\{
    \Etc[U_{n-1}^+,U_{n-1}^-,\tau_n]
    \\
    &\phantom{\le \Einit[u_0,\grid_0]+\sum_{n=1}^N\Big\{
      } +
    \Espace[U,U_{n-1}^-,\tn,\tau_n,f_n,\grid_n]+\Econs[f,t_{n-1},\tau_n]\Big\}.
  \end{align*}
  
\end{prop}
\begin{proof}
  By the decomposition of the residual \eqref{eq:resdeco} and the triangle inequality, we
  estimate on each interval $I_n$, $n=1,\ldots,N$
  \begin{align*}
    \norm[L^2(I_n;\V^*)]{\Res(\hU)}^2&\leq 3
    \norm[L^2(I_n;\V^*)]{\Rest(\hU)}^2
    +3\norm[L^2(I_n;\V^*)]{\Resh(\hU)}^2
    \\
    &\qquad
    +3     \norm[L^2(I_n;\V^*)]{f-f_n}^2.
  \end{align*}
  Now the assertion follows by Proposition \ref{p:err=res},
  \eqref{eq:rest}, and Lemma \ref{l:UpSpat}. 
\end{proof}

\begin{prop}[Lower Bound]\label{p:lower}
  Supposing the conditions of Proposition \ref{p:upper}, we have 
  \begin{align*}
    \Etc[U_{n-1}^+,U_{n-1}^-,\tau_n]&+
    \Espace[U,U_{n-1}^-,\tn,\tau_n,f_n,\grid_n]
      \\
      &\qquad\leq C\,\Big\{
      \norm[L^2(I_n;\V^*)]{\Dt(u-\hU)}^2+\norm[L^2(I_n;\V)]{u-\hU}^2
      \\
      &\qquad\qquad\quad+\int_{I_n}\oscG[\grid_n]^2(f_n, \hU)\dt+\Econs[f,t_{n-1},\tau_n]\Big\},
  \end{align*}
  where the constant $C$ depends solely on the shape regularity of
  $\grids$ and on $s$.
\end{prop}

\begin{proof}
  We first consider the spatial indicators. By Lemma \ref{l:LowSpat}
  there exists $C>0$, such that 
  \begin{align*}
    \Espace[U,U_{n-1}^-,\tn,\tau_n,f_n,\grid_n]\le
    C
    \norm[L^2(I_n;\V^*)]{\Resh(\hU)}^2+C\int_{I_n}\oscG[\grid_n]^2(f_n,
    \hU)\dt.
  \end{align*}
  The first term on the right hand side can be further estimated using
  the decomposition of the residual, the triangle inequality, and
  \eqref{eq:rest} to obtain
  \begin{align}\label{eq:3}
    \begin{split}
    \norm[L^2(I_n;\V^*)]{\Resh(\hU)} 
    &\le\norm[L^2(I_n;\V^*)]{\Res(\hU)}+
    \norm[L^2(I_n;\V^*)]{f-f_n}
    \\
    &\quad + \E{c\tau}(U_{n-1}^-,U_{n-1}^+,\tau_n).
  \end{split}
  \end{align}
  
  It remains to bound the temporal estimator. To this end, we
  introduce a nontrivial auxiliary function $\alpha\in\Ps[2s+2]$ such
  that $\alpha\perp\Ps[2s+1]$ and 
  \begin{align*}
    \int_0^1L_0^2(t)\,\alpha(t)\dt = 1,
  \end{align*}
  which is possible since $L_0^2\in\Ps[2s+2]\setminus \Ps[2s+1]$. 
  Recalling \eqref{eq:rest}, \eqref{df:rest}, and \eqref{eq:resdeco},
  we have for $\alpha_n(t)\definedas
  \alpha\big(\tfrac{t-t_{n-1}}{\tau_n}\big)$ that
  \begin{align*}
    \Etc[U_{n-1}^+,U_{n-1}^-,\tau_n]&
     = C_\tau\Enorm{U_{n-1}^+-U_{n-1}^-}^2\int_{I_n}
    L_0^2\big(\tfrac{t-t_{n-1}}{\tau_n}\big)\,\alpha_n(t)\dt
    \\
    &=C_\tau\int_{I_n}\alpha_n\,\scp[]{\Res(\hU)}{U-\hU} 
    -\alpha_n\,\scp{f-f_n}{U-\hU}\dt  
    \\
    &\quad-C_\tau
    \int_{I_n}\alpha_n\,\scp[]{\Resh(\hU)}{U-\hU}\dt.    
  \end{align*}
  The last term vanishes since
  $\Scp[]{\Resh(\hU)}{U-\hU}\in\Ps[2s+1]$. Using the Cauchy Schwarz and Young inequalities,  
  we can hence estimate
  \begin{align*}
    \Etc[U_{n-1}^+,U_{n-1}^-,\tau_n]\leq
    2C_\tau\norm[L^\infty(0,1)]{\alpha}^2\,
    \Big\{\norm[L^2(I_n;\V^*)]{\Res(\hU)}^2+\norm[L^2(I_n;\V^*)]{f-f_n}^2\Big\}. 
  \end{align*}
  Combining this with \eqref{eq:3}, we arrive at 
  \begin{align}\label{eq:4}
    \begin{split}
    \Etc[U_{n-1}^+,U_{n-1}^-,\tau_n]&+
    \Espace[U,U_{n-1}^-,t,\tau_n,f_n,\grid_n]
      \\
      &\le \Big(C
      \big(1+2\norm[L^\infty(0,1)]{\alpha}^2C_\tau\big)+2\norm[L^\infty(0,1)]{\alpha}^2C_\tau\Big)
      \\
      &\qquad \Big\{ \norm[L^2(I_n;\V^*)]{\Res(\hU)}^2+\norm[L^2(I_n;\V^*)]{f-f_n}^2\Big\}.
    \end{split}
  \end{align}
  Together with Proposition \ref{p:err=res} this is the desired estimate.
\end{proof}

\begin{rem}[Implicit Euler]
  We emphasize that the proof of the lower bound Proposition
  \ref{p:lower} is slightly different from the one in
  \cite{Verfurth:03} and yields different constants also for  
  the  \dGs[0] scheme. To see
  this, we observe that the definition of $\alpha$ implies for $s=0$
  that
  \begin{align*}
    \alpha(t)=30(6t^2-6t+1),\qquad\text{whence}\qquad
    \norm[L^\infty(0,1)]{\alpha}= 30.
  \end{align*}
  Therefore,
  we conclude for the constant in~\eqref{eq:4} with
  $C_\tau=\frac13$ from Remark~\ref{r:Ct}, that
  \begin{align*}
    C
      \big(1+2\norm[L^\infty(0,1)]{\alpha}^2C_\tau\big)+2\norm[L^\infty(0,1)]{\alpha}^2
    C_\tau=
      601\,C +600,
  \end{align*}
  where $C$ is the constant in the estimate of Lemma~\ref{l:LowSpat}.
  In contrast to this, the techniques used in \cite{Verfurth:03} for the
  implicit Euler scheme yield the 
  constant  
  \begin{align*}
    \big(1+7C_\grids^{1/2}\big)^2\,C_\grids^{1/2}\, C^{3/2}\, 12^2.
  \end{align*}
\end{rem}

\begin{rem}[Elliptic Problem]\label{r:elliptic}
  In case of the implicit Euler scheme \dGs[0], it is well known, that
  in each time-step $1\le n\le N$, $\UIn\in \Ps[0](\Vn)=\Vn$ is the
  Ritz approximation to a coercive elliptic problem. Moreover, the
  spatial estimators \eqref{df:sest} are the standard residual based
  estimators for this elliptic problem. This observation transfers to
  the \dGs scheme for $s\ge1$. To see this, we observe that (after
  transformation to the unit interval) \eqref{eq:discrete} is a
  Galerkin approximation to the solution $u_\tau\in\Ps(\V)$ of a
  problem of the kind
  \begin{align}\label{eq:elliptic}
    \begin{split}
      \int_0^1\frac1\tau\scp{\Dt u_\tau}{v} + \bilin{u_\tau}{v}\dt
      &+\frac1\tau\scp{u_\tau(0)}{v(0)}
      \\
      &= \int_0^1\scp{\bar f}{v} \dt +\frac1\tau\scp{v^-}{v(0)}
    \end{split}
  \end{align}
  for all $v\in\Ps(\V)$ and some data $\bar f\in\Ps(L^2(\Omega))$,
  $v^-\in L^2(\Omega)$, and $\tau>0$. The mappings $v\mapsto v(0)$ and
  $v\mapsto\Dt v$ are linear and continuous on $\Ps(\V)$, whence this
  equation can be taken as a vector valued
  linear variational problem of second
  order on $\V^{s+1}$.  Testing with $v=u_\tau$
  proves coercivity
  \begin{align*}
    \begin{split}
      \int_0^1\frac1\tau\scp{\Dt u_\tau}{u_\tau} + \bilin{u_\tau}{u_\tau}\dt
      &+\frac1\tau\scp{u_\tau(0)}{u_\tau(0)}
      \\
      &=\frac1{2\tau}\norm{u_\tau(0)}^2+\frac1{2\tau}\norm{u_\tau(1)}^2+
      \int_0^1\Enorm{u_\tau}^2\dt.
    \end{split}
  \end{align*}
  Obviously, its residual in $V\in\Ps(\V)$ is given by
  \begin{align*}
    \scp[]{\Resh(\mathcal{V})}{v}&=\scp[]{\bar
                                   f-\Dt\mathcal{V}}{v}-\bilin{V}{v},\quad v\in\Ps(\V),
  \end{align*}
  where $\mathcal{V}\in\Ps[s+1](\V)$ is such that
  $\mathcal{V}(c_j)=V(c_j)$, $j=1,\ldots, s$ and
  $\mathcal{V}(0)=v^-$; compare with~\eqref{eq:hU}. 
  Thanks to  Lemmas \ref{l:UpSpat} and \ref{l:LowSpat}, for
  $V\in\Ps(\VG)$, $\grid\in\grids$,  the standard 
  residual based estimator for this problem is given by 
  $\Espace[V,v^-,\tau,0,\bar f,\grid]$.
\end{rem}

\paragraph{Energy Estimation}
We shall now generalise the energy estimate from \cite{KrMoScSi:12} to
higher order $\dGs$ schemes.

\begin{prop}[Uniform global energy estimate]
  \label{prop:uniform_bound}
  Assume $N\in\N\cup\{\infty\}$ and arbitrary time instances
  $0=t_0<\cdots<t_N\le T$ with time -step-sizes
  $\tau_1,\ldots,\tau_N>0$. 
  Let $U_0=\ProG[0]u_0$ and for $1\le n\le N$ let $\UIn\in\Ps(\Vn)$
  be the discrete solutions to \eqref{eq:discrete} and let $\hU\in\W$
  as defined in \eqref{eq:hU}. 
  Then for any $m=1,\dots,N$ we have
  \begin{gather*}
    \sum_{n=1}^m \norm{\Dt\hU}^2 +
    \enorm{U_{n-1}^+-\ProG[n]U_{n-1}^-}^2 
    +\enorm{U_{n}^-}^2 -\enorm{\ProG[n]U_{n-1}^-}^2
    \le \sum_{n=1}^m\int_{I_n}\norm{f_n}^2\dt. 
  \end{gather*}
\end{prop}
\begin{proof} 
  We choose $V\definedas\ProG[n]\Dt\hU_{|I_n}\in\Ps(\Vn)$ as a test function in
  \eqref{eq:discreteb} obtaining
  \begin{align}\label{eq:1}
    \int_{I_n}\norm{\ProG[n]\Dt\hU}^2 + \bilin{
    U}{\ProG[n] \Dt\hU}\dt = \int_{I_n} \scp{f_{n}}{\ProG[n]\Dt\hU}\dt. 
  \end{align}
  In order to analyse the second term on the left hand side, we first observe
  that
  $\ProG[n]\Dt\hU_{|I_n}=\Dt\ProG[n]\hU_{|I_n}\in\Ps(\Vn)$. Recalling
  \eqref{eq:hUb} and 
  that $\B:\V\times\V\to\R$ is constant in time, we obtain integrating by
  parts, that 
  \begin{align*}
     \int_{I_n} \bilin{U}{\ProG[n] \Dt\hU}\dt&=
     - \int_{I_n} \bilin{\Dt U}{\ProG[n]\hU}\dt +
     \Enorm{\Un^-}^2-\bilin{\Un[n-1]^+}{\ProG[n]\Un[n-1]^-}. 
     \intertext{Since $\bilin{\Dt U}{\ProG[n]\hU}_{|I_n}\in\Ps[2s]$, we can apply
   \eqref{eq:RadauIIA} and conclude with \eqref{eq:hUa} that}
 \int_{I_n} \bilin{U}{\ProG[n] \Dt\hU}\dt&=
     - \int_{I_n} \bilin{\Dt U}{U}\dt +
     \Enorm{\Un^-}^2-\bilin{\Un[n-1]^+}{\ProG[n]\Un[n-1]^-}
                                           \\
 &=
     \frac12\Enorm{\Un[n-1]^+-\ProG[n]\Un[n-1]^-}^2 
    -  \frac12\Enorm{ \ProG[n]\Un[n-1]^-}^2+  \frac12\Enorm{ \Un^-}^2,
   \end{align*}
   where we used that $\bilin{\Dt \UIn}{\UIn}= \frac12\Dt \Enorm{\UIn}^2$.
   Inserting this in \eqref{eq:1} yields 
   \begin{multline*}
      \int_{I_n}\norm{\ProG[n]\Dt\hU}^2\dt+  \frac12\Enorm{\Un[n-1]^+-\ProG[n]\Un[n-1]^-}^2 
    -  \frac12\Enorm{ \ProG[n]\Un[n-1]^-}^2+  \frac12\Enorm{ \Un^-}^2
    \\
    = \int_{I_n} \scp{f_{n}}{\ProG[n]\Dt\hU}\dt. 
   \end{multline*}
   Estimating the right hand side with the help of the Cauchy-Schwarz
   and the Young
   inequality proves the assertion. 
\end{proof}

\begin{cor}
  \label{cor:uniform_bound}
 Under the conditions of Proposition \ref{prop:uniform_bound}, assume that 
  \begin{align}\label{cond:energ-cont}
    \enorm{\Un[n-1]^-}^2-\enorm{\ProG[n]\Un[n-1]^-}^2 + 
    \frac12\int_{I_n} \norm{\ProG[n]\Dt \hU}^2\dt \ge 0\quad\text{ for}\quad
    n=1,\dots,N.
  \end{align}
  Then we have the estimate 
  \begin{align*}
    \sum_{n=1}^m  \frac12\int_{I_n} \norm{\ProG[n]\Dt \hU}^2\dt +
    \enorm{\Un[n-1]^+-\ProG[n]\Un[n-1]^-}^2 
      \le \norm[\Omega\times(0,t_m)]{f}^2 +
    \enorm{ U_0}^2-\enorm{\Un[m]^-}^2. 
  \end{align*}
  In particular, the series $\sum_{n=1}^N \enorm{\Un[n-1]^+-\ProG[n]\Un[n-1]^-}^2 
  $ is uniformly bounded irrespective of the
  sequence of time-step-sizes used.
\end{cor}

\begin{proof}
  Summing up the nonnegative terms in \eqref{cond:energ-cont} yields
  \begin{gather*}
    0\le \sum_{n=1}^{m}\enorm{\Un[n-1]^-}^2-\enorm{\ProG[n]\Un[n-1]^-}^2 + 
     \frac12\int_{I_n} \norm{\ProG[n]\Dt \hU}^2\dt,
  \end{gather*}
  which is equivalent to
  \begin{gather*}
    \enorm{\Un[m]^-}^2 - \enorm{\UIn[0]}^2\le\sum_{n=1}^{m}\enorm{\Un[n-1]^+}^2-\enorm{\ProG[n]\UIn[n-1]^-}^2 + 
     \frac12\int_{I_n} \norm{\ProG[n]\Dt \hU}^2\dt.
  \end{gather*}
  Using this in the estimate of Proposition
  \ref{prop:uniform_bound} yields the desired estimate.
\end{proof}  

Having a closer look at the indicator $\E{c\tau}$ we note that, since
we allow for coarsening, it is not a pure temporal error
indicator. Coarsening may cause the loss of information and 
to few information my lead to wrong decisions within the adaptive method.
For this reason we use the triangle inequality to
split 
\begin{subequations}
  \label{eq:indicators_2}
\begin{align}
\Etc[v^-,v^+,\tau,\grid]\le \Ecoarse[v^-,\tau,\grid] + \Etime 
\end{align}
into a measure
  \begin{align}
     \label{eq:Ec}
     \Ecoarse[v^-,\tau,\grid] &\definedas\sum_{\elm\in\grid}\Ecoarse:=
     6\,C_{\tau}\sum_{\elm\in\grid}\tau\Enorm[\elm]{\ProG v^- - v^-}^2
     \intertext{for the coarsening error and}
    \label{eq:Et}
    \Etime &\definedas 6\,C_{\tau}\tau\Enorm{v^+ - \ProG v^-}^2,
  \end{align}
\end{subequations}
which serves as an indicator for the temporal error. This allows us to control the coarsening error separately.

Assuming that \eqref{cond:energ-cont} holds, Corollary~\ref{cor:uniform_bound} provides
control of the sum of the time error indicators
\begin{math}
  \Etime[{\Un[n-1]^+},{\Un[n-1]^-},\tau_n,\grid_n] = 6C_{c\tau}\tau\enorm{{\Un[n-1]^+}-\ProG[n]{\Un[n-1]^-}}^2
\end{math}.  Assumption \eqref{cond:energ-cont} would trivially be
satisfied for the Ritz-projection $R_n\Un[n-1]^-$ of $\Un[n-1]^-$ into
$\Vn$, since $\enorm{R_n\Un[n-1]^-}\le\enorm{\Un[n-1]^-}$. The
$L^2$-projection $\Pron\Un[n-1]^-$, however, does not satisfy
this monotonicity property in general and therefore coarsening may
lead to an increase of energy.  The algorithm presented below ensures
that \eqref{cond:energ-cont} is fulfilled at the end of every time-step.
To this end, using the notation~\eqref{eq:Vh}, we 
define for $V\in\Ps(\V(\grid))$, $v^-\in\V$,
$t_\star\in(0,T)$, $I=(t_\star,t_\star+\tau]\subset(t_\star,T]$, $\grid\in\grids$, and $\elm\in\grid$, the indicators
\begin{equation*}
  \Estar[V,v^-,t_\star,\tau,\grid,\elm] \definedas \enorm[\elm]{\ProG v^-}^2 - \enorm[\elm]{v^-}^2
  - \frac12\int_I \norm[E]{ \ProG\Dt \mathcal{V}}^2\dt,
\end{equation*}
as well as the convenient notation
\begin{math}
  \Estar[V,v^-,t_\star,\tau,\grid] \definedas \sum_{E\in\grid} \Estar[V,v^-,t_\star,\tau,\grid,\elm]
\end{math}.
Condition \eqref{cond:energ-cont} is then equivalent to
$\Estar[U,{\Un[n-1]^-},t_{n-1},\tau_n,\grid_n]\le 0$, $n=1,\dots,N$.
Note that the term $- \int_{I_n}\norm[E]{ \ProG \Dt \mathcal{V}}^2$ may
compensate for $\enorm[\elm]{\ProG v^-}^2 > \enorm[\elm]{v^-}^2$.

\section{The adaptive algorithm \TAFEM}
\label{sec:tafem}
Based on the observations in the previous section and a new concept for
marking we shall next describe the adaptive algorithm \TAFEM~in this section. 
In contrast to the algorithms presented in~\cite{KrMoScSi:12} and
\cite{ChenFeng:04}, the
\TAFEM is based on a different marking philosophy.
In fact, they mark according to the same indicators, 
\eqref{df:fest}-\eqref{df:sest} and \eqref{eq:indicators_2}, but 
in contrast to \cite{KrMoScSi:12,ChenFeng:04}, the \TAFEM uses an $L^2$ instead of an
$L^\infty$ strategy. 
Philosophically, this aims at an $L^2$
rather than an
$L^\infty$ equal-distribution of the error in time; compare also with
the introductory 
section~\ref{sec:introduction}. 

We follow a bottom 
up approach, i.e., we first state basic properties on some rudimentary modules
that are treated as black box routines, then describe three core modules
in detail, and finally combine these procedures in the adaptive algorithm
\TAFEM. 

\subsection{Black Box Modules}

As in~\cite{KrMoScSi:12}, we use 
standard modules \ADAPTINIT{}, \COARSEN[], \MARKREFINE{}, and \SOLVE\
as black box routines. In particular, we use the subroutine \MARKREFINE{} in an
object-oriented fashion, \ie the functionality of \MARKREFINE{} changes
according to its arguments. We next state the basic properties of these
routines.

\begin{ass}[Properties of modules]\label{ass:modules}
  We suppose that all rudimentary modules terminate with an output
  having the following properties.
  \begin{enumerate}[(1)]
  \item For a given initial datum $u_0\in L^2(\Omega)$ and tolerance $\TOLinit>0$,
    the output
    \begin{displaymath}
      (\Un[0],\gridn[0]) = \ADAPTINIT{(u_0,\gridinit,\TOLinit)}
    \end{displaymath}
    is a refinement $\gridn[0]\ge\gridinit$ and an approximation
    $\Un[0]\in\V(\gridn[0])$ to $u_0$ such that $\Einit[{u_0,\gridn[0]}]\le\TOLinit^2$.
  \item For given $g\in L^2(\Omega)$, 
    $\bar f\in\Ps(L^2(\Omega))$, $t_\star\in(0,T)$,
      $I=(t_\star,t_\star
      +\tau]\subset
    (t_\star,T]$, and $\grid\in\grids$, the output
    \begin{displaymath}
      U_{I} = \SOLVE{(g, \bar f, t, \tau, \grid)}
    \end{displaymath}
    is the solution $U_{I}\in\Ps(I,\VG)$ to the discrete elliptic problem 
    \begin{displaymath}
      \int_{I}\scp{\Dt U_{I}}{V} +
      \bilin{U_{I}}{V}\dt +\scp{U_{I}(t)}{V(t)}= \scp{g}{V} + \int_{I}\scp{\bar f}{V} \dt
    \end{displaymath}
          for all $V\in\Ps(\VG)$; compare with~\eqref{eq:discrete}. Hereby we assume exact integration and linear algebra.
  \item For a given grid $\grid\in\grids$ and a discrete function $V\in\V(\grid)$ the output
    \begin{displaymath}
      \grid_* = \COARSEN[{(V,\grid)}]
    \end{displaymath}
    satisfies $\grid_*\le\grid$.
  \item For a given grid $\grid$ and a set of indicators
    $\{\E{\elm}\}_{\elm\in\grid}$ the output 
    \begin{displaymath}
      \grid_* = \MARKREFINE{(\{\E{\elm}\}_{\elm\in\grid}, \grid)}\in\grids
    \end{displaymath}
    is a conforming refinement of $\grid$, where at least one
    element in the subset
    $\argmax\{\E{\elm}:\elm\in\grid\}\subset\grid$ has been refined.
  \item For given grids $\grid,\gridold\in\grids$ and a set of indicators
    $\{\E{\elm}\}_{\elm\in\grid}$, the output 
    \begin{displaymath}
      \grid_* = \MARKREFINE{(\{\E{\elm}\}_{\elm\in\grid}, \grid,\gridold)}\in\grids
    \end{displaymath}
    is a conforming refinement of $\grid$, where at least one element
    of the set  $\{\elm\in\grid\colon
    h_{\grid|\elm}>h_{\gridold|\elm}\}$ of
      \emph{coarsened elements (with respect to $\gridold$)}  
      is refined.
 \end{enumerate}
\end{ass}
For a more detailed description of these modules see Section~3.3.1 of~\cite{KrMoScSi:12}.

\subsection{The Core Modules}
The first core module \texttt{CONSISTENCY} controls  the consistency error
$\E{f}$.  Recalling its definition in \eqref{df:fest}, we see that the
consistency error is solely influenced by the time-step
size and can be computed without solving expensive discrete systems. 
Therefore, \texttt{CONSISTENCY} is used in the initialization of each time
step to adjust the time-step-size such that the local consistency
indicator $\Econs$ is below a local tolerance
$\tolcons$. 
It is important to notice that this module follows the classic
\emph{thresholding} algorithm, which ensures quasi-optimal order of
convergence in terms of the degrees of freedom; compare e.g. with \cite{BiDaDePe:02}.

\begin{algorithm}
  \caption{Module \texttt{CONSISTENCY} (Parameters $\sigma,\kappa_1\in(0,1)$
    and $\kappa_2>1$)}
  \label{alg:consistency}
  \baselineskip=15pt
  \flushleft
  \CONSISTENCY[(f,t,\tau,\tolcons)]
  \begin{algorithmic}[1]
    \STATE{compute $\Econs$}
    \WHILE[$\bigstar$\,enlarge $\tau$]{$\Econs < \sigma\,\tolcons^2$ and $\tau < T-t$} \label{line:CONS_IF_start}
     \STATE \label{line:CONS_IF_taun} $\tau = \min\{\kappa_2\tau, T-t\}$ 
    \STATE{compute $\Econs$}
    \ENDWHILE \label{line:CONS_IF_end}
    \WHILE[$\bigstar$\,reduce $\tau$]{$\Econs > \tolcons^2$}
    \label{line:CONS_while2_start}
    \STATE $\tau = \kappa_1\tau$
    \STATE{compute $\Econs$}
    \ENDWHILE\label{line:CONS_while2_end}
    \STATE $\bar f = f_{[t,{t+\tau}]} $
    \RETURN $\bar f,\tau$
  \end{algorithmic}
\end{algorithm}

We start with termination of the module \texttt{CONSISTENCY}.
\begin{lem}[Termination of \texttt{CONSISTENCY}]
  \label{lem:termination-consistency}
  Assume $f \in L^2((0,T);L^2(\Omega))$.  Then for any $t\in(0,T)$ and $\tauin\in(0,T-t]$, 
  \begin{displaymath}
    (\bar f,\tau)=\texttt{CONSISTENCY}(f,t,\tauin,\tolcons)
  \end{displaymath}
  terminates
  and
  \begin{align}\label{Cons:ineqtol}
    \Econs \le \tolcons^2.
  \end{align}
\end{lem}
\begin{proof}
The proof is straightforward since $\Econs$ is monotone non-increasing and  $\Econs \rightarrow 0$ when  $\tau\rightarrow 0$.
\end{proof}

Obviously, a local control of the form~\eqref{Cons:ineqtol}  does
  not guarantee, that the global consistency error is below some prescribed tolerance $\TOLcons$.
  For this reason, we first precompute some local tolerance
  $\tolcons$ from the global tolerance $\TOLcons$  by the following module $\TOLFIND$.
\begin{algorithm}[H]
  \caption{\TOLFIND (Parameters: $\tilde\tau_0$)}
  \label{alg:TOLFIND}
  \baselineskip=15pt
  \flushleft
  \TOLFIND$(f,T,\TOLcons)$
  \begin{algorithmic}[1]
    \STATE initialize   $N_f$ and  set
    $\tolcons=\TOLcons$, 
    $\tilde t_0 =0$,
    \label{lines:tolfind-start-init}
    \LOOP\label{line:tol-loop}
    \STATE $\epsilon=n=0$
    \REPEAT\label{line:time-loopTF}
    \STATE $n = n+1$
    \STATE $(f_n,\tilde\tau_n) = \CONSISTENCY[({f,\tilde t_{n-1},\tilde\tau_{n-1}}, \tolcons)]$
    \STATE $\epsilon =\epsilon+ \E{f}^2(f,\tilde t_{n-1},\tilde\tau_n)$
    \UNTIL{$\tilde t_n=\tilde t_{n-1}+\tilde\tau_n <
      T$}   \label{line:time-loop-endTF}
    \STATE $N_f=n$ 
    \IF{$\epsilon >\frac12 \TOLcons^2$}\label{line:tf-second-case}
            \STATE  $\tolcons^2=\frac12\tolcons^2$ 
                   \ELSE
           \STATE  
           \textbf{break}\COMMENT{$\bigstar$\,std.~exit}\label{line:TOLFIND-break3} 
    \ENDIF
    \ENDLOOP \label{line:tol-loop-end}
    \STATE $\tolcons^2=\min\{\tolcons^2 ,\  \frac{\TOLcons^2}{2N_f}\}$ \label{line:tol-def}
    \RETURN $\tolcons$
  \end{algorithmic}
\end{algorithm}

The next result states that if all local consistency indicators 
are  below the threshold $\tolcons$ then
the accumulation of the consistency indicators stays indeed below the
prescribed global tolerance $\TOLcons$.

\begin{lem}[Termination of \texttt{TOLFIND}]
  \label{lem:termination-tolfind}
  Assume $f \in L^2((0,T);L^2(\Omega))$.  Then for any $\TOLcons>0$,
  we have that 
  \begin{displaymath}
    \tolcons=\texttt{TOLFIND}(f,T,\TOLcons)>0
  \end{displaymath}
  terminates. Moreover, let $0= t_0< t_1<\cdots<  t_N=T$ be arbitrary 
  with $\tau_n=\tn-\tn[n-1]$, $n=1,\ldots,N$, then 
  \begin{align}\label{Cons:ineq}
    \Econs[f,{\tn[n-1]},\tau_n]\le \tolcons^2,~n=1,\ldots, N  \quad\Rightarrow\quad
    \sum_{n=1}^N \Econs[f,{\tn[n-1]},\tau_n] \le \TOLcons^2.
  \end{align}
\end{lem}

\begin{proof}
  The proof is divided into three steps. 

  \step{1} We show that the process from lines~\ref{line:time-loopTF}
  to~\ref{line:time-loop-endTF} terminates. To this end, we recall the
  parameters $\sigma,\kappa_1\in(0,1)$ and $\kappa_2>1$ from
  $\CONSISTENCY[(f,{\tilde t_{n-1}},\tilde\tau_{n-1},\tolcons)]$. 
  We argue by contradiction
  and  assume that an infinite monotone sequence $\{\tilde t_n\}_{n\ge0}
  \subset[0,T]$ is constructed by \TOLFIND~with
  $\lim_{n\to\infty}\tilde t_n=t^\star\in(0,T]$.

Let us first assume that
  $t^\star <T$, and let $\ell_0,m_0\in\N$ such that $\kappa^{\ell_0}_2\ge
  \kappa_1^{-m_0}\ge \kappa_2$.   Then there exists $n_0\in\N$, such that 
  \begin{align}\label{eq:ell_0}
    t^\star+\kappa^{\ell_0}_2
    \tilde\tau_{n}<T\qquad\text{and}\qquad
    \Econs[f,{\tilde t_n},\kappa^{\ell_0-1}_2
    \tilde\tau_{n} ]\le \sigma\tolcons^2
  \end{align}
  for all $n\ge n_0$ since $\tilde\tau_{n}\to 0$ and 
  \begin{align*}
    \Econs[f,{\tilde t_n},\kappa^{\ell_0-1}_2
    \tilde\tau_{n} ]\le
    \norm[\Omega\times(\tilde t_n,\tilde t_n+\kappa^{\ell_0-1}_2
   \tilde\tau_{n})]{f}^2\to 0\quad\text{as}~n\to\infty.
  \end{align*}
  Therefore, from the loops in lines \ref{line:CONS_IF_start} to
  \ref{line:CONS_IF_end} and in \ref{line:CONS_while2_start} to
  \ref{line:CONS_while2_end} of $\CONSISTENCY[]$, we conclude
  that 
  $\tilde\tau_{n_0+1}\ge
  \kappa_2^{\ell_0}\kappa^{m_0}_1\tilde\tau_{n_0}\ge
  \tilde\tau_{n_0}$. Indeed, we have by~\eqref{eq:ell_0} and 
   \begin{align*}
  \Econs[f,{\tilde t_{n_0}},\kappa^{\ell}_2\kappa^{m_0}_1
    \tilde\tau_{n_0} ]\le
    \Econs[f,{\tilde t_{n_0}},\kappa^{\ell-1}_2\tilde\tau_{n_0} ]\le
     \sigma\tolcons^2
  \end{align*}
  that $\ell\ge \ell_0$ and $m\le m_0$.
  Consequently, we have 
  $\tilde\tau_{n}\ge \tilde\tau_{n_0}$, for all $n\ge n_0$ by induction.
  This is the contradiction.

  Let now $t^\star=T$, then with similar arguments as before, we conclude
  that $\Econs[f,{\tilde t_n},T-{\tilde t_n}]\le \sigma\tolcons^2$ for some
  $n\in\N$, and we have from line~\ref{line:CONS_IF_taun} of
  \CONSISTENCY[], that
  $\tilde\tau_{n}=T-\tilde t_n$, which contradicts the assumption in this case.

\step{2} We next check that the condition of line~\ref{line:tf-second-case}
is violated after finite many steps. 
Since the span of characteristics of dyadic intervals is
dense in $L^2(0,T)$, we  can choose $M>0$, such that the squared 
consistency error on
the grid of $2^M$ uniform intervals is below $\frac14 \TOLcons^2$. 
We split the intervals generated in $\TOLFIND(f,T,\tolcons)$ into 
\begin{align*}
\I_{in}:=\big\{n: (\tilde t_{n-1},\tilde t_n]\subset T(m2^{-M},(m+1)2^{-M}]~\text{for some}~m\in\{0,\ldots,2^M-1\}\big\}
\end{align*}
and $\I_{out}:=\{1,\ldots,N_f\}\setminus\I_{in}$
according to whether or not they are included in one of the dyadic
intervals. Therefore, we have, with the monotonicity of the consistency
error and $\#\I_{out}\le 2^M$, that 
\begin{align*}
  \epsilon =
  \sum_{n\in\I_{in}}\Econs[f,{\tilde t_{n-1}},\tilde\tau_{n}]+\sum_{n\in\I_{out}}\Econs[f,{\tilde
  t_{n-1}},\tilde\tau_{n}]\le
  \frac14\TOLcons^2 + 2^M\tolcons^2.
\end{align*}
Taking $\tolcons^2< 2^{-(M+2)}\TOLcons^2$, we see that the condition of
line~\ref{line:tf-second-case} is violated, which proves the assertion.

\step{3} Combining the above steps, we conclude that \TOLFIND \
terminates and it remains to prove~\eqref{Cons:ineq}. To this end, we
proceed similarly as in \step{2} and let 
\begin{align*}
  \I_{in}:=\big\{n: (t_{n-1},t_n]\subset (\tilde t_{m-1},\tilde
  t_m]~\text{for some}~m\in\{1,\ldots,N_f\}\big\}. 
\end{align*}
and $\I_{out}:=\{1,\ldots, N\}\setminus\I_{in}$. 
By monotonicity, we have $\sum_{n\in\I_{in}}\Econs[f,{ t_{n-1}},\tau_{n}]\le
\sum_{n=1}^{N_f}\Econs[f,{ \tilde t_{n-1}},\tilde
\tau_{n}]\le\TOLcons^2/2$ and thus the assertion follows from
\begin{align*}
  \sum_{n=1}^N\Econs[f,{t_{n-1}},\tau_{n}]&=
  \sum_{n\in\I_{in}}\Econs[f,{ t_{n-1}},\tau_{n}]+\sum_{n\in\I_{out}}\Econs[f,{
  t_{n-1}},\tau_{n}]
                                            \\
  &\le \frac{\TOLcons^2}2 + N_f\tolcons^2 = \frac{\TOLcons^2}2 +
    N_f\frac{\TOLcons^2}{2N_f}\le \TOLcons^2.\qedhere
\end{align*}
\end{proof}

\begin{rem}[Estimation of $\tolcons$ under regularity assumptions]\label{rem:underregular}
Supposing the regularity assumption $f \in
H^s((0,T);L^2(\Omega))$, $s\in(0,1]$, the following idea may be used
as an alternative for the estimation of $\tolcons$ with \TOLFIND.

Let $\delta>0$. Then using Lemma~\ref{lem:termination-consistency}
together with  \Poincare's inequality in $H^s$ and the fact that there
are at most $ \frac{T}{\delta}$ disjoint intervals of length $\delta$
in $(0,T]$, we obtain
\begin{align*}
  \sum_{n=1}^N\Econs[f,{t_{n-1}},\tau_{n}]&=
  \sum_{\tau_n>\delta}\Econs[f,{ t_{n-1}},\tau_{n}]+\sum_{\tau_n\leq\delta}\Econs[f,{
  t_{n-1}},\tau_{n}]
                                            \\
  &\le \frac{T}{\delta}\tolcons^2 +\sum_{\tau_n\leq\delta} \tau_n^{2s}
    \norm[H^s(t_{n-1},t_n,L^2(\Omega))]{f}^{2} 
  \\ 
 &= \frac{T}{\delta}\tolcons^2 + \delta^{2s}  \norm[H^s(0,T,L^2(\Omega))]{f}^{2}.
\end{align*}
 By choosing $\delta= \left(\frac{T \, \tolcons}{
     \norm[H^s(0,T,L^2(\Omega))]{f}}\right)^{\frac{2}{2s+1}}$, the
 previous estimate turns into 
 \begin{align*}
  \sum_{n=1}^N\Econs[f,{t_{n-1}},\tau_{n}]&\leq  2\,T^{\frac{2s}{2s+1}} \, \norm[H^s(0,T,L^2(\Omega))]{f}^{\frac{2}{2s+1}}  \tolcons^{\frac{4s}{2s+1}}.
\end{align*}
In other words, if a priori knowledge 
on the regularity of the right hand side is available then
\TOLFIND \ can be replaced by the somewhat simpler term
$$
\tolcons^2 =
2^{-\frac{2s+1}{2s}}\,T^{-1}
\norm[H^s(0,T,L^2(\Omega))]{f}^{-\frac{1}{s}}
\,\TOLcons^{\frac{2s+1}{s}}.$$ 
\end{rem}

We turn to the module \texttt{ST\_ADAPTATION}, listed in
Algorithm~\ref{alg:TS_ADAPTATION}, which handles a single time-step.
The module adapts the grid and the time-step-size according to the 
indicators involving the discrete solution of the current time-step,
namely the space indicator $\E{\grid}$ and 
the separated coarsening and time
indicators  $\E{c}$ and $\E{\tau}$. 
The routine requires right at te start of each iteration the computation of the discrete solution on
the actual grid and with the current time-step-size; see
line~\ref{line:TSA-solve}. Note that in
 \texttt{ST\_ADAPTATION} only refinements are performed (both in space and in
time). 
Recalling the discussion in the introductory section,
Section~\ref{sec:introduction}, we aim to use a thresholding algorithm
for the  indicators $\E{\tau}$, in order to equally distribute the
time error. To this end, we first need to guarantee $\E{*}\le
0$ in order to control the global time error 
with the help of the uniform energy estimate form
Corollary~\ref{cor:uniform_bound}. Since for neither  the
space nor the coarsening errors  there is a similar control available, 
we relate the corresponding indicators to the time
or the consistency indicator, i.e. to
adapt the spatial triangulation until 
\begin{align}\label{eq:l2strat}
\E{c}^2,\E{\grid}^2\le
\E{\tau}^2+\E{f}^2.
\end{align}
Here we have invoked the consistency indicator $\E{f}$ on the right
hand side although it is controlled 
by \texttt{CONSISTENCY} outside \texttt{ST\_ADAPTATION}
-- note that $\E{f}$ does not depend on the discrete solution. In
fact, from the uniform
energy estimate, Corollary~\ref{cor:uniform_bound}, we have that 
$\E{\tau}$ vanishes faster than $\E{f}$ by one order, when no additional regularity of $f$
is assumed. Consequently, the time-step size is dictated by $\E{f}$,
which may leed to $\E{\tau}\ll\tolst$. Thanks to
Lemma~\ref{lem:termination-tolfind}, we expect that \eqref{eq:l2strat}
leads to an equal distribution of the errors in time
in most cases. However, the case
$\max\{\E{\tau},\E{f}\}\ll\min\{\tolst,\tolcons\}$ cannot be avoided
theoretically, hence we have accomplished~\eqref{eq:l2strat}
with a safeguard $L^\infty$ marking; compare with
lines \ref{line:TSA-Espace-condition}
and~\ref{line:TSA-Ecoarse-condition} of \texttt{ST\_ADAPTATION}.

Note that in the above discussion, we have concentrated
  on an equal distribution in time and have tacitly assumed that
  in each time-step the local space indicators are optimally
  distributed, which is motivated by the optimal convergence analysis
  for elliptic problems; compare e.g.  with \cite{Stevenson:07,CaKrNoSi:08,DiKrSt:16}.

\begin{algorithm}
  \caption{Module \texttt{ST\_ADAPTATION} (Parameter
    $\kappa\in(0,1)$
   )}
  \label{alg:TS_ADAPTATION}
  \baselineskip=15pt
  \flushleft{\STADAPTATION[(U_t^-, f,  t, \tau, \grid,\gridold,\tolst)]}
  \algsetup{indent=1.5em}
  \begin{algorithmic}[1]
    \STATE compute $\Econs$
    \LOOP  
    \STATE $I=[t,t+\tau]$
    \STATE $\bar f =  f_I$
    \STATE $U_{I} = \SOLVE{(U_t^-, \bar f, t,\tau,
      \grid)}$ \label{line:TSA-solve}
    \STATE $U_t^+=\lim_{s\searrow t}U_I(s)$
    \STATE compute $\{\Espace[U_{I},U_t^-,t,\tau,\bar
    f,\grid,\elm]\}_{\elm\in\grid}$, $\{\Estar[{U_t^+,
      U_t^-,\tau,\grid,E}]\}_{\elm\in\grid}$
   $\Etime[U_t^+,U_t^-,\tau,\grid]$,  and
    $\{\Ecoarse[U_t^-,\tau,\grid,E]\}_{E\in\grid}$    \label{line:TSA-est-space-tc} 
    \IF
    {$ \Etime[U_t^+,U_t^-,\tau,\grid] >
      \tolst^2
      $} 
    \label{line:TSA-Etime-condition}
    \STATE $\tau =
    \kappa\tau$ \COMMENT{\boxed{\textsf{A}}} \label{line:TSA-adapttau}
    \STATE compute $\Econs$
    \ELSIF
    {$ \Espace[U_{I},U_t^-,t,\tau,\bar
      f,\grid] > 
     \Etime[U_t^+,U_t^-,\tau,\grid]+\Econs+ \tau \tolst
     $}  \label{line:TSA-Espace-condition}
    \STATE $\grid =
    \MARKREFINE{(\{\Espace[U_{I},U_t^-,t,\tau,\bar f,\grid,\elm]\}_{\elm\in\grid},
      \,\grid)}$ \COMMENT{\boxed{\textsf{B}}}
    \ELSIF
    {$ \Ecoarse[U_t^-,\tau,\grid]>
      \Etime[U_t^+,U_t^-,\tau,\grid] +\Econs+ \tau \tolst
      $}  
    \label{line:TSA-Ecoarse-condition}
    \STATE $\grid=\MARKREFINE{(\{\Ecoarse[U_t^-,\tau,\grid,\elm]\}_{\elm\in\grid},
      \,\grid)}$ \COMMENT{\boxed{\textsf{C}}}
      \ELSIF
    {$\Estar[{U_t^+, U_t^-,\tau,\grid}]>0$} \label{line:TSA-Estar-condition}
    \STATE $\grid = \MARKREFINE{(\{\Estar[{U_t^+,
        U_t^-,\tau,\grid,\elm}]\}_{\elm\in\grid},
      \grid,\gridold)}$\COMMENT{\boxed{\textsf{D}}} 
    \ELSE
    \STATE \textbf{break}\COMMENT{$\bigstar$\,exit\,}\label{line:TSA-break1}
    \ENDIF  
     \ENDLOOP \label{line:TSA-outer-end}
    \RETURN $U_{I},\tau,\bar f,\grid$
  \end{algorithmic}
\end{algorithm}

\begin{rem}\label{rem:custimisation}
  We note that the \textbf{if} conditions in
  lines~\ref{line:TSA-Espace-condition}
  and~\ref{line:TSA-Ecoarse-condition} of \STADAPTATION[] may involve additional
  parameters. For instance, line \ref{line:TSA-Ecoarse-condition} may be
  replaced by 
  \begin{algorithmic}[1]
    \setcounter{ALC@line}{14}
    \STATE{\textbf{else if}
    {$  \Ecoarse[U_t^-,\tau,\grid]\ge 
     \gamma_{c}\,
     \Etime[U_t^+,U_t^-,\tau,\grid] + \rho_c\Econs+\sigma_c \tau \tolst$}
   \textbf{then}}  
  \end{algorithmic}
  with $\gamma_c,\rho_c,\sigma_c>0$ and similar for the space
  indicator $\est_\grid$ in line~\ref{line:TSA-Espace-condition} with constants $\gamma_\grid,\rho_\grid,
  \sigma_\grid>0$.  This requires some modifications of the \TAFEM, 
  which would make the presentation more technical.  
  For the sake of clarity of the presentation, we decided to skip
  these customisation possibilities; compare also with
  Remark~\ref{rem:custimisation2}.  
\end{rem}

\subsection{The main module \TAFEM}

We are now in the position to formulate the \TAFEM in
Algorithm~\ref{alg:TAFEM} below. 

In the initialization phase
the given tolerance $\TOL>0$ is split into 
tolerances 
$\TOLinit,\TOLcons,\TOLst>0$.  Next, \ADAPTINIT{} provides a
sufficiently good approximation $\Un[0]$ of the initial datum $u_0$.
Then the time-step iteration is entered, where each single time-step
consists of the following main steps.  We first initialize the time-step
size by \texttt{CONSISTENCY} and then conduct one coarsening step with
\texttt{COARSEN}.
The adaptation of the grid and time-step-size with respect to the indicators for
the spatial, temporal, and coarsening error is done by
\texttt{ST\_ADAPTATION}.
\begin{algorithm}[H]
  \caption{\TAFEM}
  \label{alg:TAFEM}
  \baselineskip=15pt
  \begin{algorithmic}[1]
    \STATE initialize $\grid_\text{init}$, $\tau_0$ and set $t_0 =0$,
    $n=0$ \label{lines:astfem-start-init}
    \STATE split tolerance $\TOL>0$ such that
    \begin{math}
      \TOLinit^2 + 
      3\,\TOLcons^2 + \TOLst^2 = \TOL^2 \label{TAFEM:splitTOL}
    \end{math}
    \STATE $\tolcons=\TOLFIND(f,T,\TOLcons)$
    \STATE $(U_0^-,\grid_0) = \ADAPTINIT{(u_0, \gridinit,\TOLinit)}$
    \label{lines:astfem-end-init}
    \STATE compute $C_T:= 
    6\,\sqrt{6 \,C_{c\tau} \, T}  \left( \norm[\Omega\times(0,T)]{f}^2 +
    \enorm{ U_0^-}^2 \right)^\frac12   + 2\,T $\label{TAFEM:C_T}
    \REPEAT\label{line:time-loop}
    \STATE $n = n+1$
    \STATE $\tau_n=\min\{\tau_{n-1},T-t\}$
    \STATE $\tau_n = \CONSISTENCY[({f,t_{n-1},\tau_{n-1}}, \tolcons)]$
    \label{line:astfem-end-I} 
    \STATE $\gridn =  \COARSEN[{(U_{n-1}^-,\gridn[n-1])}]$ \label{line:astfem-start-I}
    \STATE $(U_{|I_n}, \tau_n, f_n,\gridn) = \STADAPTATION[{(\Un[n-1]^-,t_n,\tau_n,
      f,
      \gridn,\gridn[n-1],\TOLst^2/C_T)}]$  \label{line:astfem-IIa}
    \STATE $U_{n}^-=U_{|I_n}(t_{n-1}+\tau_n)$
    \UNTIL{$t_n=t_{n-1}+\tau_n < T$}
  \end{algorithmic}
\end{algorithm}

\section{Convergence}
\label{sec:convergence}
In this section, we first prove that the core modules and \TAFEM
terminate and then verify that the estimators and thus the error is below the
given tolerance. Throughout the section we suppose that the black-box
modules satisfy Assumption~\ref{ass:modules}.

Before turning to the main module \texttt{ST\_ADAPTATION}, as an
auxiliary result, we shall
consider convergence of the adaptive finite element method for 
stationary elliptic problems of the kind~\eqref{eq:elliptic}, which have to be solved
in each timestep.
\begin{algorithm}
  \caption{\texttt{AFEM}}
  \label{alg:AFEM}
  \flushleft{\texttt{AFEM}$(v^-, \bar f,  t, \tau, \grid^0)$}
  \algsetup{indent=1.5em}
  \begin{algorithmic}[1]
    \STATE set $k=0$
    \LOOP
    \STATE  $U_\tau^k = \SOLVE{(v^-, \bar f, 0,\tau,
      \grid^k)}$
    \STATE  compute $\{\Espace[U_\tau^k,v^-,0,\tau,\bar
    f,\grid,\elm]\}_{\elm\in\grid}$,
    \STATE $\grid^{k+1} =
    \MARKREFINE{(\{\Espace[U_\tau^k,v^-,0,\tau,\bar
      f,\grid^k,\elm]\}_{\elm\in\grid}, 
      \,\grid^k)}$
    \STATE $k=k+1$
    \ENDLOOP
  \end{algorithmic}
\end{algorithm}

\begin{prop}[Convergence for the Elliptic
  Problem]\label{P:adapt-ellipt}
  Suppose that $U_t^- \in L^2(\Omega)$, $\bar f\in\Ps(L^2(\Omega))$, and
  $\tau>0$. Then, starting from any grid $\grid^0\in\grids$ we have
  for the sequence $\{\grid^k ,U_\tau^k\}_{k\ge 0}\subset \grids\times
  \Ps(\V)$ generated by \texttt{AFEM}$(v^-,\bar f,t,\tau,\grid^0)$, that 
  \begin{displaymath}
    \Espace[U_\tau^k,v^-,\tau,t,\bar f,\grid^K] \rightarrow
    0\quad\text{as}~k\to\infty. 
  \end{displaymath}
\end{prop}

\begin{proof}
  Recalling Remark~\ref{r:elliptic}, we have that
  $\Espace[U_\tau^k,v^-,\tau,0,\bar f,\grid^K]$ are the standard
  residual based a
  posteriori error estimators for the coercive
  problem~\eqref{eq:elliptic}. From Lemmas~\ref{l:UpSpat}
  and~\ref{l:LowSpat} and Assumption~\ref{ass:modules} on
  \MARKREFINE{}, we have that the conditions of \cite[Theorem
  2.2]{Siebert:11} are satisfied. This yields the assertion.
\end{proof}

\begin{lem}[Termination of \texttt{ST\_ADAPTATION}]
  \label{lem:termination-tsa}
  For any $t\in(0,T)$, $\tauin\in(0,T-t]$, $\grid,\gridold\in\grids$,
  and $U_t^-\in\V(\gridold)$, we have that  
  \begin{displaymath}
    (U_I,\tau,\bar f,\grid)=\STADAPTATION[{(U_t^-,f,t,\tauin,\gridin,\gridold,
      \tolst)}]
  \end{displaymath}
  terminates. 
  Moreover, we have $\grid\ge\grid_0$, $\Estar[{U_t^+, U_t^-,\tau,\grid}]\le0$,
  \begin{align*}
    &\tauin\ge \tau \ge  
    \min\left\{\tauin,\frac{\kappa\,\tolst^2}{ 6 \big(
    \norm[\Omega\times(t,t+\tau)]{f}^2 + 
    \enorm{ U_t^-}^2 \big)}\right\},
  \end{align*}
  and the indicators satisfy the tolerances
  \begin{align*}
    \Etime[U_t^+,U_t^-,\tau,\grid] &\le
    \tolst^2,
    \\
    \Espace[U_{I},U_t^-,t,\tau,\bar f,\grid]&\leq
        \Etime[U_t^+,U_t^-,\tau,\grid]
      +\Econs  + \tau \tolst,
    \\ 
     \Ecoarse[U_t^-,\tau,\grid]&\leq 
        \Etime[U_t^+,U_t^-,\tau,\grid]+
      \Econs  + \tau \tolst,
  \end{align*}
  where $U_t^+=\lim_{s\searrow t}U_{(t,t+\tau]}(s)$.
\end{lem}

\begin{proof}
  In each iteration of the loop in \texttt{ST\_ADAPTATION} at first, a discrete
  solution $U_I$ is computed on
  the current grid $\grid$ with the actual time-step size
  $\tau$.
  Then either the time-step-size is
  reduced or the actual grid is refined. More precisely, exactly one of
  the statements labeled as
  \boxed{\textsf{A}},\dots,\boxed{\textsf{D}} in
  Algorithm~\ref{alg:TS_ADAPTATION} is executed, any of them terminating
  by Assumption~\ref{ass:modules}.  Whenever one of these statements is
  executed the corresponding indicator is positive.

  In statement \boxed{\textsf{C}} 
  the grid is refined due to the coarsening indicator $\E{c}$.
  Thanks to Assumption~\ref{ass:modules}~(5), after a finite number of
  executions of \boxed{\textsf{C}}, a grid $\grid$ is obtained with
  $\gridold\le\grid$ and thus $\Ecoarse[U_t^-,\grid]=0$,
  i.e. statement \boxed{\textsf{C}} is not entered anymore. 
  This happens irrespective of 
  refinements in other statements. 

  In statement \boxed{\textsf{D}} 
  the grid is refined with respect to the indicators $\est_*$
  controlling the energy gain due to coarsening. Therefore, it follows
  from the same reasoning as for statement \boxed{\textsf{C}}, that
  statement \boxed{\textsf{D}} is also executed at most until the coarsening is fully
  removed after finite many refinement steps.

It is important to notice that if statement \boxed{\textsf{A}} is
executed then the conditions in
lines~\ref{line:TSA-Estar-condition} and~\ref{line:TSA-Etime-condition}
imply
   \begin{align*}
    \frac{1}{\tau} \leq  \frac{1}{\tau}
     6\,\tau\enorm{U_t^+-\ProG U_t^-}^2 \frac{1}{\tolst^2} \leq 
        \frac{1}{\tolst^2}  6\, \left(
     \norm[\Omega\times(t,t+\tau)]{f}^2 +
    \enorm{ U^-_t}^2 \right),
  \end{align*}
where the last inequality follows from
Corollary~\ref{cor:uniform_bound}. 
This implies that $\tau$ is bounded from below and thus Statement \boxed{\textsf{A}} is only executed
  finite many times. This also proves the asserted lower bound on
  the final time-step size.

  Assuming that \texttt{ST\_ADAPTATION} does not terminate, we 
  infer from the fact that all other statements are only conducted finitely many times,
  that statement \boxed{\textsf{B}} has to be executed infinite many
  times. In other words, the loop reduces to the adaptive iteration \texttt{AFEM}
  with fixed data $U_t^-$, $\bar f$, $t$, and $\tau$. Therefore, 
  Proposition~\ref{P:adapt-ellipt} contradicts the condition
     in
  line~\ref{line:TSA-Espace-condition}.

  In summary, we deduce that \texttt{ST\_ADAPTATION} terminates and the
  iteration is abandoned in line \ref{line:TSA-break1}. This proves
  the assertion.
\end{proof}

We next address the termination of the main module \TAFEM.

\begin{prop}[Termination of \TAFEM]\label{P:termination_tafem}
  The adaptive algorithm \TAFEM terminates for any initial 
  time-step-size $\tau_0>0$ and produces a finite number of time instances
  $0=t_0<\dots<t_N=T$.
  
  Moreover, we have 
  $\Einit[{u_0,\gridn[0]}]\le\TOLinit^2$ and that the consistency error complies with
  \eqref{Cons:ineq}. For all $n=1,\ldots,N$, we have that the
  estimates in Lemma~\ref{lem:termination-tsa} are satisfied with
  $t=t_{n-1}$, $\tau=\tau_n$, $U_I=U_{|I_n}$, $U_t^\pm=U_{n-1}^\pm$,
  $\grid=\grid_n$, and $\gridold=\grid_{n-1}$.
\end{prop}
\begin{proof}
  Each loop starts with setting the time-step-size such that
  $\tau_n\leq T-t_n$, $n\in\N$. 
  Thanks to Assumption~\ref{ass:modules} for the black-box modules,
  Lemma~\ref{lem:termination-consistency} for \CONSISTENCY[], and
  Lemma~\ref{lem:termination-tsa} for \texttt{ST\_ADAPTATION},
  all modules of \TAFEM terminate and in each timestep the asserted
  properties are satisfied.

  Since we have $\Estar[{U_{n-1}^+, U_{n-1}^-,\tau_n,\grid_n}]\le0$
  for all $n$, we may conclude 
  $\enorm{U_{n-1}^-}\le \norm[\Omega\times(0,T)]{f}^2+\enorm{U_0}$
  from Lemma~\ref{cor:uniform_bound} and
  thus it follows with Lemma~\ref{lem:termination-tsa}, that
  \begin{align*}
    \tauin_n\ge \tau_n \ge  
    \min\Big\{\tauin_n,\frac{\kappa\,\tolst^2}{12 \big(
    \norm[\Omega\times(0,T)]{f}^2 + 
    \enorm{ U_0}^2 \big)}\Big\},
  \end{align*}
  where $\tauin_{n}=\CONSISTENCY[({f,t_{n-1},\tau_{n-1}},
  \tolcons)]$. Assuming that the final time is not reached
  implies $\tau_n\to 0$ as $n\to\infty$ and therefore there exists
  $n_0\in\N$, such that $\tau_n=\tauin_{n}$ for all $n\ge n_0$. 
  Now, the contradiction follows as in step~\step{1} of the proof of
  Lemma~\ref{lem:termination-tolfind}. 
  \end{proof}

Collecting the results derived above allows us to prove the main result.

\begin{thm}[Convergence into Tolerance]\label{Thm:main}
  Algorithm \TAFEM computes for any prescribed tolerance $\TOL>0$ and
  initial time-step-size $\tau_0>0$ a partition $0<t_0<\dots<t_N=T$
  with associated meshes $\{\gridn\}_{n=0,\dots,N}$, such that we have
  for the 
  corresponding approximation $\hU\in\W$ from~\eqref{eq:discreteb},
  that
  \begin{align*}
    \Wnorm{u-\hU} \le \TOL.
  \end{align*}
\end{thm}
\begin{proof}
  Thanks to
  Proposition~\ref{P:termination_tafem}, we have that \TAFEM
  terminates and it remains to prove the
  error bound. For the sake of brevity of the presentation, we shall
  use the abbreviations
  \begin{alignat*}{2}
    \Etime[n]&:= \Etime[{\Un[n-1]^+,\Un[n-1]^-,
      \tau_n,\grid_n}],&\qquad \Econs[n]&:= \Econs[f,t_{n-1},\tau_n],
    \\
    \Espace[n]&:=\Espace[{U,\Un[n-1]^-,t_{n-1},\tau_n, f_n,\gridn}],
    &\quad\text{and}\quad
    \Ecoarse[n]&:= \Ecoarse[{\Un[n-1]^-,\tau, \gridn}].
  \end{alignat*}
 
  The initial error satisfies $\Einit[{u_0,\gridn[0]}]\le\TOLinit^2$ by
  by Assumption~\ref{ass:modules}. Thanks to the choice of the precomputed local tolerance
  $\tolcons$,  we know from Lemma~\ref{lem:termination-tolfind} 
  that the consistency error is bounded by
  $\TOLcons$, i.e. we have \eqref{Cons:ineq}.  

  When finalizing a time-step, we also have from Lemma~\ref{P:termination_tafem} that 
  \begin{align*}
    \Etime[n] &\le \tolst^2
     \qquad
                \text{and}
                 \qquad
      \Espace[n], \Ecoarse[n]\leq
        \Etime[n]+\Econs[n] +\tau_n \tolst,
    \end{align*}
  with $\tolst=\TOLst^2/C_T$.
  Combining this with \eqref{eq:indicators_2} and \eqref{Cons:ineq},
  we conclude 
  \begin{align*}
    \sum_{n=1}^{N}\Espace[n]
    +\,\Etc[{\Un[n-1]^+},{\Un[n-1]^-,\tau_n}]
      &\le \sum_{n=1}^{N}\Espace[n]
    +\Ecoarse[n]
    +\Etime[n]
      \\
      &\le \sum_{n=1}^{N} 
      2\,\tau_n \tolst+2\,\Econs[n]+ 3\, 
      \Etime[n] 
      \\
      &\le2\,T \, \tolst + 2\,\TOLcons^2+3
      \sum_{n=1}^{N}
        \Etime[n].
  \end{align*}
  Using Corollary~\ref{cor:uniform_bound} for the last term, we get
  for any $\delta>0$, that 
    \begin{align*}
     \sum_{n=1}^{N}  
    \Etime[n]
      &= \sum_{\tau_n>\delta}
      \Etime[n] +
      \sum_{\tau_n\leq\delta} \Etime[n]
      \\
      & \leq \frac{T}{\delta} \tolst^2 + \delta \sum_{n=1}^N
      6\, C_{\tau}\,\Enorm{\Un[n-1]^+ -\Pi_{\grid_n}\Un[n-1]^-}^2 \\
      &\leq \frac{T}{\delta} \tolst^2 + \delta 6 \,C_{\tau} \left(
        \norm[\Omega\times(0,T)]{f}^2 + \enorm{ U_0}^2 \right)
    \end{align*}
  and by choosing 
  \begin{equation*}
    \delta = \left( \frac{T}{ 6 \,C_{c\tau}  \left( \norm[\Omega\times(0,T)]{f}^2 +
    \enorm{ U_0}^2 \right)}\right)^{\frac{1}{2}}  \tolst,
  \end{equation*}
  we obtain 
  \begin{align*}
    \sum_{n=1}^{N}  
    \Etime[n]
    &\leq 2\left( 6\,C_{c\tau} \,T  \left( \norm[\Omega\times(0,T)]{f}^2 +
    \enorm{ U_0}^2 \right) \right) ^\frac12 \tolst .
  \end{align*}
  Inserting this into the above estimate yields 
   \begin{multline*}
    \sum_{n=1}^{N}\Espace[n]
    +\,\Etc[{\Un[n-1]^+},{\Un[n-1]^-,\tau_n}]
      \\
      \begin{split}
        &\le \underbrace{\left( 
          6\,\sqrt{6\,C_{c\tau}\,T}
           \left( \norm[\Omega\times(0,T)]{f}^2 + \enorm{
              U_0}^2 \right)^\frac12+ 2\,T \right)}_{=C_T}\, \tolst+2\,\TOLcons^2
        \\
        &\le\TOLst^2+2\,\TOLcons^2.
      \end{split}
  \end{multline*}

  Collecting the bounds for the indicators $\E{0}$, $\E{\grid}$,
  $\E{c\tau}$, and $\E{f}$, recalling the splitting 
  \begin{displaymath}
    \TOLinit^2 + 
    3\,\TOLcons^2+ \TOLst^2  = \TOL^2,
  \end{displaymath}
  and taking into account the upper bound of
  Proposition~\ref{p:upper} proves the assertion.
\end{proof}

\begin{rem}\label{rem:custimisation2}
  In order to guarantee the main result
  (Theorem~\ref{Thm:main}) also for the modifications of
  Remark~\ref{rem:custimisation} line~\ref{TAFEM:C_T} in  \TAFEM must
  be changed to  
   \begin{algorithmic}[1]
    \setcounter{ALC@line}{4}
    \STATE compute $C_T:= 
    (1+\gamma_c+\gamma_\grid)\,2\,\sqrt{6 \,C_{c\tau} \, T}  \left( \norm[\Omega\times(0,T)]{f}^2 +
    \enorm{ U_0^-}^2 \right)^\frac12   + (\sigma_c+\sigma_\grid)\,T $.  
  \end{algorithmic}
  Moreover, the splitting
  of the tolerances in line~\ref{TAFEM:splitTOL} must be changed to 
  \begin{algorithmic}[1]
    \setcounter{ALC@line}{1}
     \STATE split tolerance $\TOL>0$ such that
    $\TOLinit^2+(1+\rho_\grid+\rho_c)\TOLcons^2+\TOLst^2=\TOL^2$.
   \end{algorithmic}
\end{rem}

\section{Numerical aspects and experiments\label{sec:numerics}}
We conclude the article by illustrating some practical aspects of the
implementation with three numerical experiments. We compare the
presented algorithm \TAFEM with the algorithm \ASTFEM introduced in
\cite{KrMoScSi:12}. 

\subsection{The implementation}
The experiments are implemented in DUNE~\cite{DUNE:16} using the 
DUNE-ACFEM (http://users.dune-project.org/projects/dune-acfem) module. The computations utilize
linear conforming finite elements on space and 
$\dGs[0]$ as time-stepping scheme.  All simulations where performed on a Intel\textregistered Core\texttrademark i7-6700HQ Processor with 64 GB RAM.

Both algorithms \TAFEM and \ASTFEM start from exactly the same initial
mesh $\gridinit$. 
The initial values are interpolated on the mesh and local refinements
are performed in order to comply with the initial tolerance. 
On the resulting meshes the needed constants are computed  (the
minimal time-step size $\tau^*$  for \ASTFEM and  $\tolcons$ from
{\TOLFIND} for \TAFEM ). 

In order to control $\est_c$ and $\est_*$, the algorithms need to
handle two meshes and corresponding finite 
element spaces at every new time-step. This is realised
exploiting the tree structure of the
refinements of macro elements as
in~\cite{KrMoScSi:12}.  
At every new time-step all elements on the mesh are marked to be
coarsened up to two times and then adapted again if necessary. 
The mentioned estimators are computed only up to constants and used
for the adaptive refinement progress.
The spatial marking relies on the equi-distribution strategy, which
marks every element with an estimator bigger than the
arithmetic mean. 

The following remark lists the tolerance splitting used by \ASTFEM.
\begin{rem}\label{R:tol-split}
  In \cite{KrMoScSi:12}, the \ASTFEM uses the tolerance splitting 
  \begin{align*}
    \TOL^2={\TOL}_0^2+T\widetilde{\TOL}_f^2+T\widetilde{\TOL}_{\grid\tau}^2+\widetilde{\TOL}_*^2.
  \end{align*}
  Thereby $\TOL_*^2$ is used to compute a minimal safeguard step-size
  $\tau_*$. 
  The method computes then an approximation $\mathcal{U}\in\W$ to~\eqref{eq:weak}, such
  that 
  \begin{align*}
    \Einit[u_0,\grid_0]\le \TOL_0^2,\qquad  \sum_{n=1}^N\Big\{\Econs[f,t_{n-1},\tau_n]\Big\}&\le T\widetilde{\TOL}_f^2
                         \intertext{and}
    \sum_{n=1}^N\Big\{
    \Etc[U_{n-1}^+,U_{n-1}^-,\tau_n]
    +
    \Espace[U,U_{n-1}^-,\tn,\tau_n,f_n,\grid_n]\Big\}
    &\le T\widetilde{\TOL}_{\grid\tau}^2+\widetilde{\TOL}_*^2.
  \end{align*}
  This motivates the relation 
  \begin{align*}
    T\widetilde{\TOL}_f^2=3{\TOL}_f^2,\quad \text{and}\quad
    T\widetilde{\TOL}_{\grid\tau}^2+\widetilde{\TOL}_*^2 =\TOL_{\grid\tau}^2
  \end{align*}
  in the examples below.
  
  For the simulations presented below we have used the following
  comparable splittings for the two methods \ASTFEM and \TAFEM
  relative to the total tolerance $\TOL$: 
  \begin{itemize}
  \item $\TOL_0^2=\frac1{10}\TOL^2$,
  \item  $\TOL_f^2=T\widetilde{\TOL}^2_f =\frac4{10} \TOL^2$,
  \item  $\TOL_{\grid\tau}^2=T\widetilde{\TOL}_{\grid\tau}^2+\widetilde{\TOL}_*^2 =\frac6{10} \TOL^2$,
   \item  $\widetilde{\TOL}_*^2 =\frac1{100} \TOL^2$.
  \end{itemize}
\end{rem}

\subsection{The experiments}
In this section, we introduce the three numerical experiments in
detail and discuss the numerical results. 
\subsubsection{Singularity in time}\label{sec:singtime}
This numerical experiment is constructed on the spatial domain
$\Omega=(0,1)^2\subset\R^2$ over the time interval $(0,T)=(0,2)$ with
homogeneous Dirichlet boundary conditions and homogeneous initial
data. The right-hand side $f$ is choosen such that the exact solution
is given by 
\begin{equation*}
u(\boldsymbol{x},t) = |t-\bar{t}|^\alpha\sin(\pi(x^2-x)t)\sin(\pi(y^2-y)t)
\end{equation*}
with parameters $\bar{t}=\frac{\pi}{3}$ and $\alpha=0.7$. The graph of
$u$ has a singularity in time at $t=\frac{\pi}{3}$. Hence, the
right-hand side contains the therm
$\operatorname{sgn}(t-\bar{t})\alpha|t-\bar{t}|^{\alpha-1}$. A direct
calculation shows that this term is $L^2$ -integrable but is not in
$H^1$. 
This particular example shows one main advantage of \TAFEM.
In fact, in contrast to \ASTFEM, \TAFEM does not require
the right hand side $f$ to have temporal derivative in $L^2$ in order
to control the consistency error $\E{f}$.

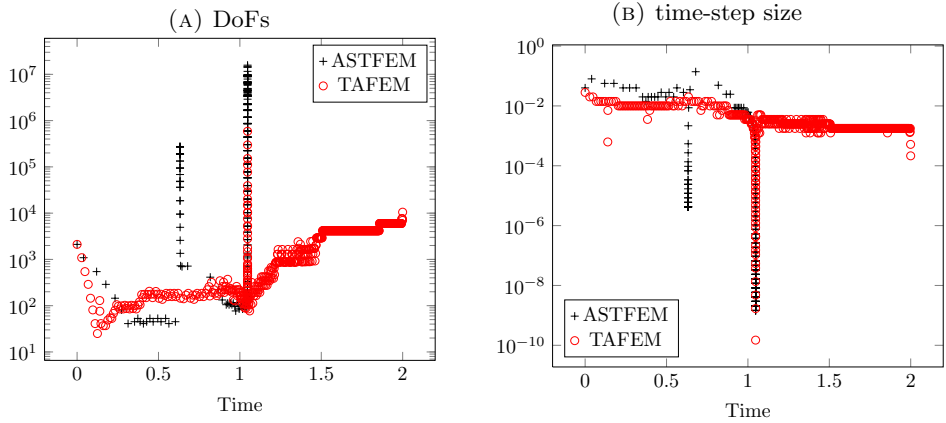
\begin{figure}[h]
\begin{subfigure}[c]{0.45\textwidth}
\subcaption{DoFs}
\begin{tikzpicture}[scale=0.75]
\begin{semilogyaxis}[xlabel={Time}]
\addplot[black, only marks, mark=+] table[ x=Time, y=Dofs ] {ts-logfile-astfem.dat};
\addlegendentry{\ASTFEM}
\addplot[red, only marks, mark=o] table[x=Time, y=Dofs ] {ts-logfile-tafem.dat};
\addlegendentry{$\TAFEM$}
\end{semilogyaxis}
\end{tikzpicture}
\end{subfigure}\qquad
\begin{subfigure}[c]{0.45\textwidth}
\subcaption{time-step size}
\begin{tikzpicture}[scale=0.75]
\begin{semilogyaxis}[xlabel={Time},legend style={legend pos=south west}]]
\addplot[black, only marks, mark=+] table[ x=Time, y=DeltaT ] {ts-logfile-astfem.dat};
\addlegendentry{\ASTFEM}
\addplot[red, only marks, mark=o] table[x=Time, y=DeltaT ] {ts-logfile-tafem.dat};
\addlegendentry{$\TAFEM$}
\end{semilogyaxis}
\end{tikzpicture}
\end{subfigure}
\caption{DoFs and time-step sizes for the singularity in time problem. }
\label{fig:ts-DoF+TSS}
\end{figure}

\begin{figure}{h}
  \centering
     \begin{tikzpicture}[scale=1]
\begin{semilogyaxis}[width = 10cm, height=8cm,xlabel={Time},legend style={legend pos=south west,font=\tiny}]
\addplot[red, solid, mark={}] table[ x=Time, y expr=\thisrow{SpaceEstimate2}+\thisrow{CT-Estimate2}+\thisrow{D-Estimate2} ] {ts-logfile-tafem.dat};
\addlegendentry{\TAFEM $\E{}^2$}
\addplot[black, dashed, thick, mark={}] table[x=Time, y expr=\thisrow{DeltaT}*(\thisrow{SpaceEstimate2}+\thisrow{CT-Estimate2}+\thisrow{D-Estimate2})] {ts-logfile-astfem.dat};
\addlegendentry{\ASTFEM $\E{}^2$}
\addplot[blue, dotted, thick, mark={}] table[x=Time, y expr=\thisrow{SpaceEstimate2}+\thisrow{CT-Estimate2}+\thisrow{D-Estimate2}] {ts-logfile-astfem.dat};
\addlegendentry{\ASTFEM $\frac1\tau \E{}^2$}
\end{semilogyaxis}
\end{tikzpicture}
  \caption{The local error estimators
    $\E{\tau}^2+\E{\grid}^2+\E{c}^2+\E{f}^2$ for \TAFEM and \ASTFEM as
    well as the sum of local $L^\infty$ indicators
    $\frac1\tau(\E{\tau}^2+\E{\grid}^2+\E{c}^2+\E{f}^2)$ used by \ASTFEM
    for the singularity in time problem.}
\label{fig:ts-Etaf+East}
\end{figure}
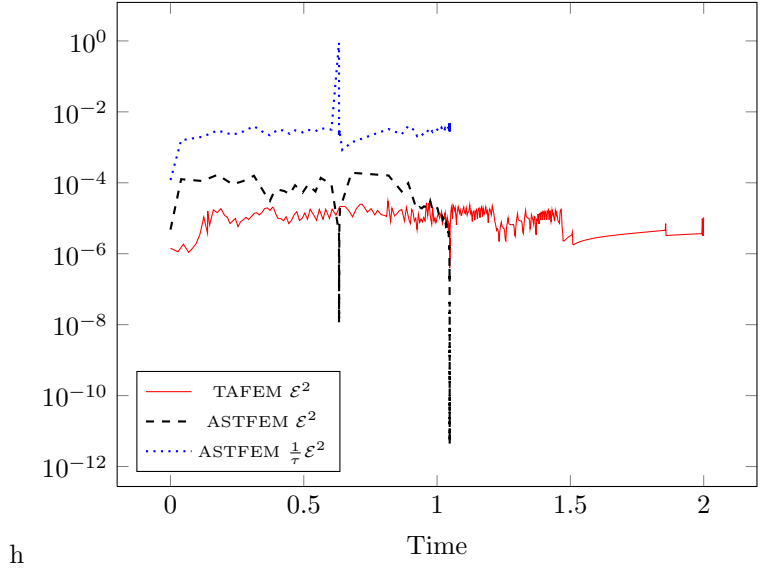

\footnotesize
\begin{figure}[h]
\begin{tabular}{ccccc}
  time  & time-step \ASTFEM &  DoFs \ASTFEM &   time-step \TAFEM & DoFs \TAFEM  \\ \hline
  1.0  & 0.00613614 & 97 & 0.00353598  & 193\\
  1.02  & 0.00433893  & 85 & 0.00353601 & 112\\
  1.03  & 0.0030681 & 97 & 0.00250034  &  109\\
  1.04  & 0.00153406 &  97 & 0.00125018   & 109\\
  1.042  & 0.00108475 & 97 & 0.00125018   & 125\\
  1.044  & 0.00108475  & 97 & 0.000884012  & 132\\
  1.045  & 0.000542376 & 157 & 0.000625091   & 128\\
  1.046 & 0.000383519 & 193 & 0.000312546  & 242\\
  1.047 & 3.3899e-05 & 713 & 7.81367e-05  & 448\\
  1.0471 & 1.19852e-05 & 2073 & 3.90684e-05  & 635\\
  1.0472 & 9.36409e-08  & 226082 & 1.7266e-06  & 3693\\
  1.0473 & $non$ & $non$ & 3.90691e-05  & 622\\
\end{tabular}
\caption{Time-steps and DoFs of \ASTFEM and \TAFEM for the singularity in time problem.}\label{T:TS-stepsize}
\end{figure}
\normalsize

\ASTFEM was killed after time-step 288 in which $14\, 163\, 460$ DoFs
are used as well as a time-step size of  $1.46338\mathrm{e}{-9}$.  As can be
observed from 
Figure~\ref{fig:ts-DoF+TSS}, \ASTFEM massively 
refines in time and space. It was killed before reaching the singularity at
$\bar{t}=\frac{\pi}{3}$, thereby accumulating
the total number of  $498\, 228\, 711$ DoFs. The reason for this
behaviour lies in the $L^\infty$ marking. Indeed,  
Figure~\ref{fig:ts-Etaf+East} shows that \ASTFEM equally distributes
the $L^\infty$ indicators thereby leading to very small local errors,
which cause the strong spatial refinement. Note
that the minimal step-size $\tau_*$ in \ASTFEM only applies when
temporal refinement is performed due to the time indicator
$\E{\tau}$,  i.e., time-step sizes below the threshold can
$\tau_*$ be chosen when required by the consistency estimator
$\est_f$, which is the case close to the singularity. Consequently,
the behaviour of \ASTFEM cannot essentially improved by a different
choice of $\TOL_*$.
In contrast, the local estimators \TAFEM appears to be quite equally distributed. 
It uses slightly larger time-steps and by far less DoFs close to the
singularity; compare with the table of Fig.~\ref{T:TS-stepsize}. It completely outperforms \ASTFEM and  
reaches the final time with a total of $2\, 947\, 080$ DoFs in 618 time-steps.

\subsubsection{Jumping singularity}
Inspired by example $5.3$ of~\cite{MoNoSi:00}, we construct an
experiment where the solution has a strong spatial singularity that
changes its position in time. In the domain $\Omega\times(0,4]$,
with $\Omega=(0,3)\times(0,3)$, we define the  elliptic operator $\L
u=-\divo{\Am\nabla u}$, where   
\begin{equation*}
  \Am(t,x) =
  \begin{cases}
    a_1\mathbb{I}\qquad & \text{if } (x-x_i)(y-y_i)\geq0\\
    a_2\mathbb{I}\qquad & \text{if } (x-x_i)(y-y_i)<0
  \end{cases}
\end{equation*}
with $a_1=161.4476387975881$, $a_2=1$, $i= \lceil t \rceil$,
$(x_1,y_1)=(1,2)$, $(x_2,y_2)=(1,1)$, $(x_3,y_3)=(2,1)$, and
$(x_4,y_4)=(2,2)$. This operator will `move' the singularity through
the points $x_i$.  
Let $u$ be the function
\begin{equation*}
  u(x,t) = \sum_{i=1}^4 s_i(t)\ r_i^\gamma\ \mu(\theta_i)
\end{equation*}
where
\begin{equation*}
  s_i(t) =
  \begin{cases}
    (t-(i-1))^2(t-i)^2\quad &\text{if } i-1\leq t \leq i\\
    0\qquad & \text{otherwise}
  \end{cases}
\end{equation*}
and
\begin{equation*}
  \mu(\theta) =
  \begin{cases}
    \cos((\frac{\pi}{2}-\sigma)\gamma)\cdot\cos((\theta-\frac{\pi}{2}+\rho)\gamma) \quad&  \text{if } 0\leq\theta<\frac{1}{2}\pi\\
    \cos(\rho\gamma)\cdot\cos((\theta-\pi+\sigma)\gamma) &  \text{if } \frac{1}{2}\pi\leq\theta<\pi\\
    \cos(\sigma\gamma)\cdot\cos((\theta-\pi-\rho)\gamma) &  \text{if } \pi\leq\theta<\frac{3}{2}\pi\\
    \cos((\frac{\pi}{2}-\rho)\gamma)\cdot\cos((\theta-\frac{3\pi}{2}-\sigma)\gamma) &  \text{if } \frac{3}{2}\pi\leq\theta<2\pi\\
  \end{cases}
\end{equation*}
with $\gamma=0.1$, $\rho=\frac{\pi}{4}$, $\sigma=-14.92256510455152$, $x-x_i=r_i\cos(\theta_i)$ and $y-y_i=r_i\sin(\theta_i)$. 
It is easy to check that $u$ satisfies
\begin{equation*}
  \partial_t u(x,t) +\L u(x,t) = \sum_{i=1}^4 r_i^\gamma \mu(\theta_i)\ \partial_t s_i(t)\ .
\end{equation*}

 Based on the ideas presented in  Remark~\ref{R:tol-split} we compare 
\TAFEM and \ASTFEM with the same  tolerance
$\TOL=0.007$.

\begin{figure}
\begin{subfigure}[c]{0.45\textwidth}
\subcaption{DoF}
\begin{tikzpicture}[scale=0.75]
\begin{semilogyaxis}[xlabel={Time}]
\addplot[black, dashed, thick, mark={}] table[ x=Time, y=Dofs ] {js-logfile-astfem.dat};
\addlegendentry{\ASTFEM}
\addplot[red, solid, mark={}] table[x=Time, y=Dofs ] {js-logfile-tafem.dat};
\addlegendentry{$\TAFEM$}
\end{semilogyaxis}
\end{tikzpicture}
\end{subfigure}\qquad
\begin{subfigure}[c]{0.45\textwidth}
\subcaption{time-step size}
\begin{tikzpicture}[scale=0.75]
\begin{semilogyaxis}[xlabel={Time}]
\addplot[black, dashed, thick, mark={}] table[ x=Time, y=DeltaT ] {js-logfile-astfem.dat};
\addlegendentry{\ASTFEM}
\addplot[red, solid, mark={}] table[x=Time, y=DeltaT ] {js-logfile-tafem.dat};
\addlegendentry{$\TAFEM$}
\end{semilogyaxis}
\end{tikzpicture}
\end{subfigure}
\caption{DoFs and time-step sizes for the jumping singularity
  problem.}
\label{fig:js-DoF+TSS}
\end{figure}
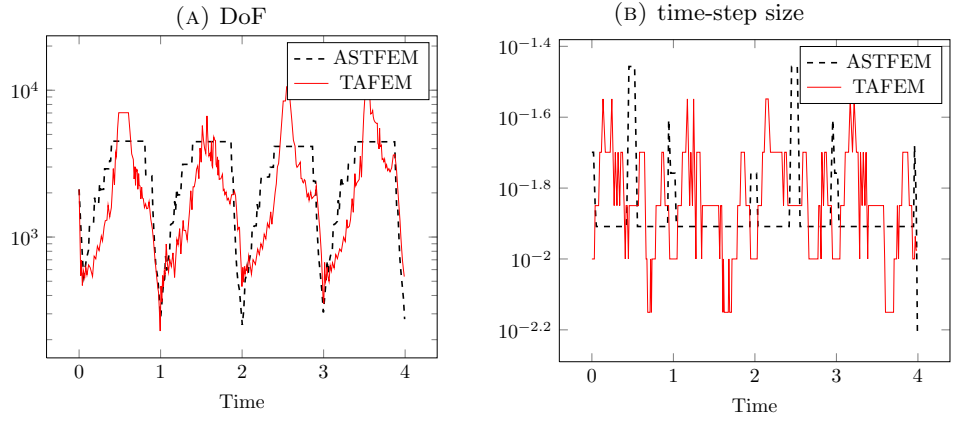

\begin{figure}
\begin{subfigure}[c]{0.22\textwidth}
\subcaption{$t\approx 0.5$}
\includegraphics[width=1.0\textwidth]{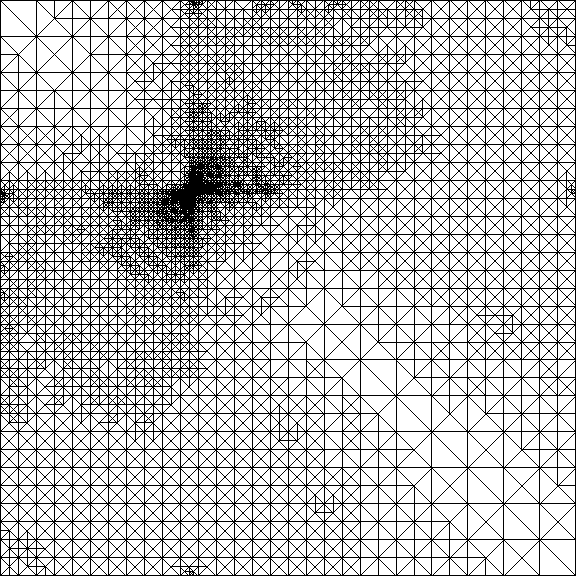}
\end{subfigure}
\quad
\begin{subfigure}[c]{0.22\textwidth}
\subcaption{$t\approx 1.0$}
\includegraphics[width=1.0\textwidth]{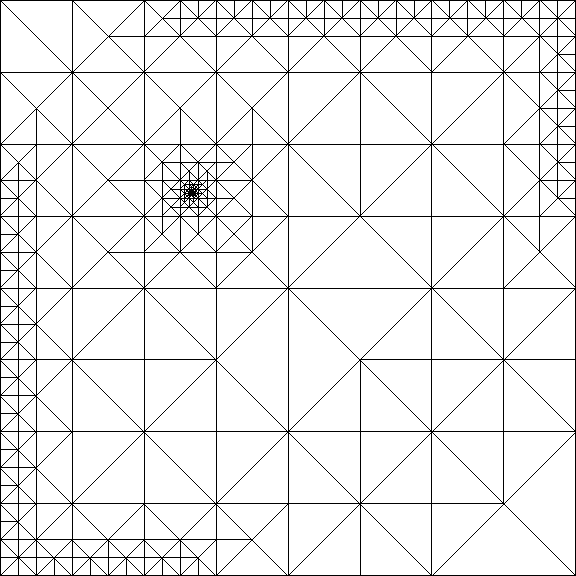}
\end{subfigure}
\quad
\begin{subfigure}[c]{0.22\textwidth}
\subcaption{$t\approx 1.5$}
\includegraphics[width=1.0\textwidth]{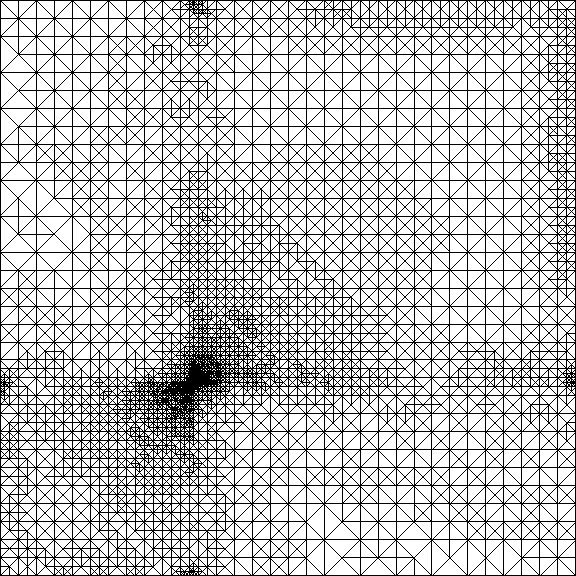}
\end{subfigure}
\quad
\begin{subfigure}[c]{0.22\textwidth}
\subcaption{$t\approx 2.0$}
\includegraphics[width=1.0\textwidth]{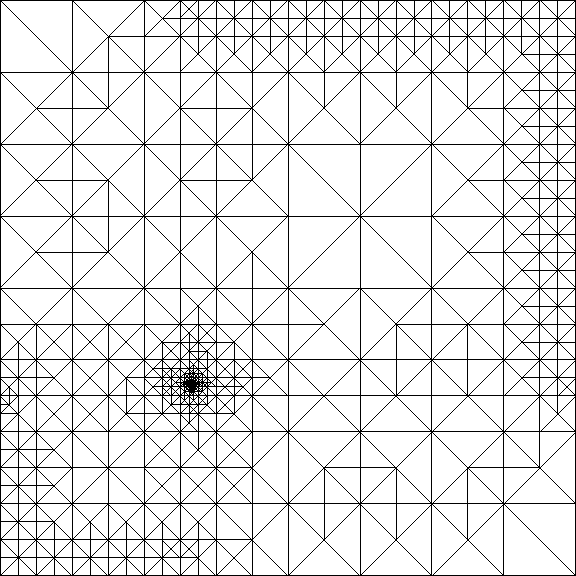}
\end{subfigure}

\begin{subfigure}[c]{0.22\textwidth}
\subcaption{$t\approx 2.5$}
\includegraphics[width=1.0\textwidth]{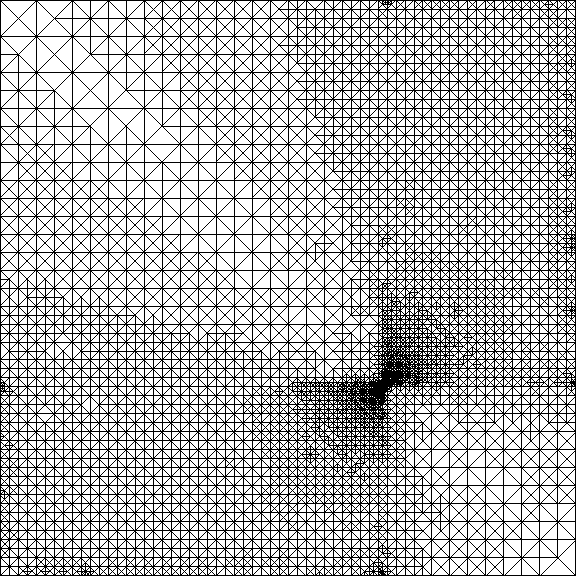}
\end{subfigure}
\quad
\begin{subfigure}[c]{0.22\textwidth}
\subcaption{$t\approx 3.0$}
\includegraphics[width=1.0\textwidth]{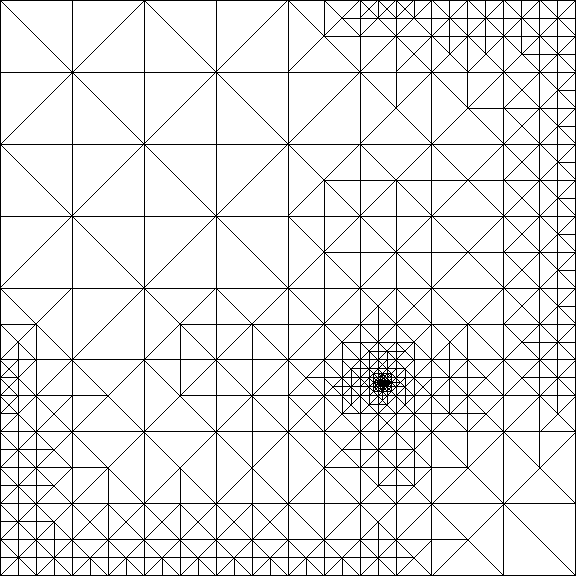}
\end{subfigure}
\quad
\begin{subfigure}[c]{0.22\textwidth}
\subcaption{$t\approx 3.5$}
\includegraphics[width=1.0\textwidth]{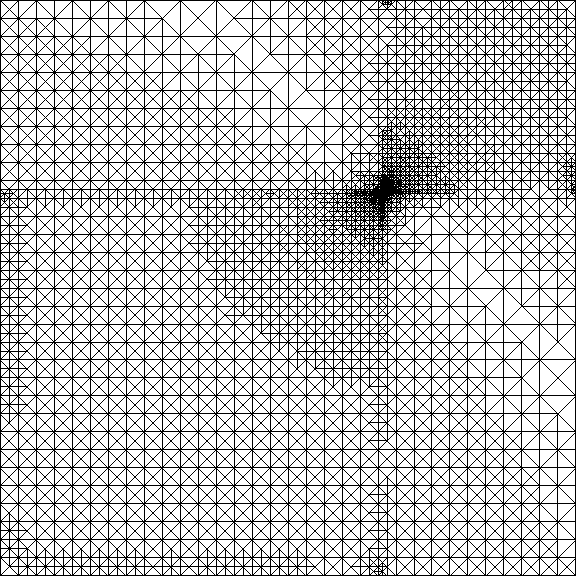}
\end{subfigure}
\quad
\begin{subfigure}[c]{0.22\textwidth}
\subcaption{$t\approx 4.0$}
\includegraphics[width=1.0\textwidth]{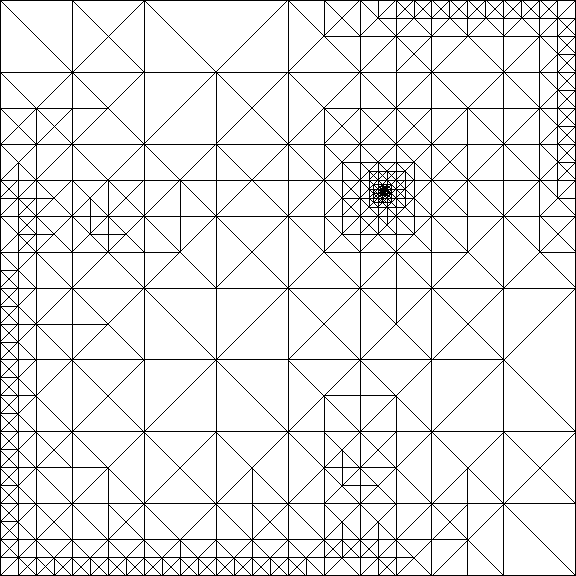}
\end{subfigure}
\caption{Adaptive grids for the jumping singularity problem.}
\label{fig:adaptMesh}
\end{figure}

\ASTFEM makes excessive use of the nonstandard exit, i.e., 
the time-step sizes equal minimal
time-step size $\tau_\ast = 0.0123477$ for 276 of a total of 302
time-steps,
and  uses a
total of 893771 Dofs. The $L^2-H^1$-error is $0.0546689$, the  $L^2-L^2$-error is  $0.0355061$ and the total computation time was 413.67 seconds.

The \TAFEM uses  a
total of 786789 Dofs in 291 time-steps. The
$L^2(0,4,H^1(\Omega))$-error is $0.0552438$, the
$L^2(0,4,L^2(\Omega))$-error is  $0.034989$  and the total computation
time was 546.113 seconds (including \TOLFIND).  
The adaptive meshes generated by \TAFEM are displayed in Figure~\ref{fig:adaptMesh}.
We see that the spatial adaptivity captures the position of the
singularity by local refinement and coarsens the region when the
singularity has passed by. 
By having a look on Fig.~\ref{fig:js-DoF+TSS} we see, that \TAFEM
makes more use of the spatial and temporal adaptivity  and 
achieves a similar result with slightly less effort.

The advantages of \TAFEM come fully into their own in the presence of
singularities in time (see Section~\ref{sec:singtime}). For regular
(in time) problems, \TAFEM is expected to perform
similar to \ASTFEM up to the disadvantage that, at the beginning, the module {\TOLFIND}
needs several adaptive iterations over 
the time span, whereas the computation for the minimal time-step size
in \ASTFEM only iterates once over the time. This is reflected in the
comparable computing times for the jumping singularity problem.

\subsubsection{Rough initial data} We conclude with an example
inspired on the numerical experiment~5.3.2  
in~\cite{KrMoScSi:12} with homogeneous Dirichlet boundary
conditions and homogeneous right-hand side $f\equiv0$. As initial data we choose a checkerboard
pattern over $\Omega=(0,1)^2$ where $u_0\equiv-1$ on $\Omega_1=(\frac{1}{3},\frac{2}{3})\times
\left(  (0,\frac{1}{3})\cup(\frac{2}{3},1) \right)\ \cup \ \left(  (0,\frac{1}{3})\cup(\frac{2}{3},1) \right)
\times (\frac{1}{3},\frac{2}{3})$, $u_0\equiv1$ on 
$\Omega\setminus\Omega_1$ and $u_0\equiv0$ on $\partial\Omega$. 
Starting with an initial mesh with only 5 DoFs, the approximation of
$u_0$ uses Lagrange interpolation and refines 
the mesh until $\|U_0-u_0\|_\Omega^2\leq\TOLinit^2 = 10^{-2}$
is fulfilled. Starting \ASTFEM and \TAFEM 
with a tolerance of $\TOL=10^{-1}$ and running to the final
time $T=1$, we get the following results: \ASTFEM needs 811 time-steps, a total sum of $436\, 199\, 377$ DoFs, with an estimated total error
of 0.0230905, and a  total computation time of 81466.4 seconds. 
The \ASTFEM makes  use of the
nonstandard exit for the first 270 time-steps, with
minimal time-step size of $\tau_\ast=7.77573\mathrm{e}{-7}$, the small
size of the time-steps in the beginning is also accompanied by extreme
spatial refinements contributing to the large total number of
DoFs. This is due to the 
$L^\infty$-strategy that aims in an equal distributing of the time-indicators
$\frac1\tau \E{\tau}^2$ rather then $\E{\tau}^2$. In order to highlight this effect close to
the initial time, we have used a log scale for the time in
Figures~\ref{fig:ts-Ect-comp} and~\ref{fig:cb-DoF+TSS}.  
The \TAFEM only needs 117 time-steps and a total of $3\, 762\, 503$ DoFs
resulting in an estimated total error of 0.039855. It is about $20$
times faster with a  total computation time of 3903.76 seconds (including \TOLFIND). 
 \TAFEM refines the mesh initially and then almost steadily coarsens in
 time and space (see Figures~\ref{fig:cb-DoF+TSS} and~\ref{fig:adaptMesh2}~(E-H)). 
Figure~\ref{fig:ts-Ect-comp} shows that the time indicators
$\E{\tau}^2$ are nearly equally distributed. 
Both algorithms reduce the spatial resolution once the singular behaviour of the solution is
reduced; see Figures~\ref{fig:cb-DoF+TSS} and~\ref{fig:adaptMesh2}.

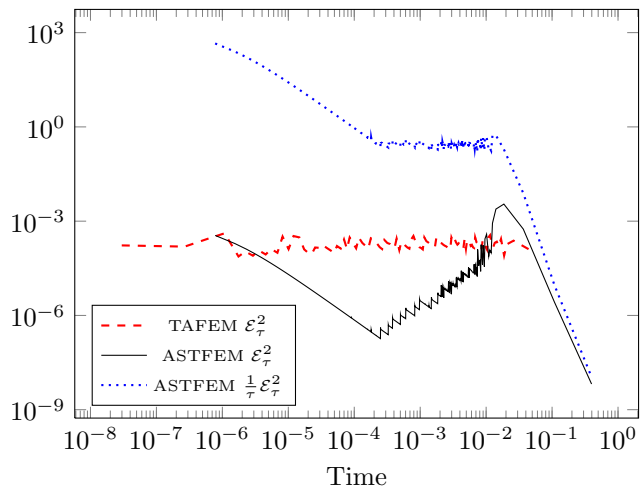
\begin{figure}
  \centering
     \begin{tikzpicture}
\begin{loglogaxis}[width = 9cm, height=7cm,xlabel={Time},legend
  style={legend pos=south west,font=\tiny}]
\addplot[red, dashed, thick,  mark={}] table[ x=Time, y expr=\thisrow{CT-Estimate2} ] {cb-logfile-tafem.dat};
\addlegendentry{\TAFEM $\E{\tau}^2$}
\addplot[black, solid, mark={}] table[x=Time, y expr=\thisrow{DeltaT}*\thisrow{CT-Estimate2}] {cb-logfile-astfem.dat};
\addlegendentry{\ASTFEM $\E{\tau}^2$}
\addplot[blue, dotted, thick, mark={}] table[x=Time, y expr=\thisrow{CT-Estimate2}] {cb-logfile-astfem.dat};
\addlegendentry{\ASTFEM $\frac1\tau \E{\tau}^2$}
\end{loglogaxis}
\end{tikzpicture}
  \caption{The local time indicator
    $\E{\tau}^2$ for \TAFEM and \ASTFEM as
    well as the local $L^\infty$ indicators
    $\frac1\tau\E{\tau}^2$ used by \ASTFEM
    for the rough
    initial data problem.}
\label{fig:ts-Ect-comp}
\end{figure}

\begin{figure}
\begin{subfigure}[c]{0.45\textwidth}
\subcaption{DoF}
\begin{tikzpicture}[scale=0.75]
\begin{loglogaxis}[xlabel={Time},legend style={legend pos=south west,font=\tiny}]
\addplot[black, solid, mark={}] table[ x=Time, y=Dofs ] {cb-logfile-astfem.dat};
\addlegendentry{\ASTFEM}
\addplot[red, dashed, thick, mark={}] table[x=Time, y=Dofs ] {cb-logfile-tafem.dat};
\addlegendentry{$\TAFEM$}
\end{loglogaxis}
\end{tikzpicture}
\end{subfigure}\qquad
\begin{subfigure}[c]{0.45\textwidth}
\subcaption{time-step size}
\begin{tikzpicture}[scale=0.75]
\begin{loglogaxis}[xlabel={Time},legend style={legend pos=north west,font=\tiny}]
\addplot[black, solid, mark={}] table[ x=Time, y=DeltaT ] {cb-logfile-astfem.dat};
\addlegendentry{\ASTFEM}
\addplot[red, dashed, thick, mark={}] table[x=Time, y=DeltaT ] {cb-logfile-tafem.dat};
\addlegendentry{$\TAFEM$}
\end{loglogaxis}
\end{tikzpicture}
\end{subfigure}
\caption{DoFs and time-step sizes for the rough initial data problem.}
\label{fig:cb-DoF+TSS}
\end{figure}
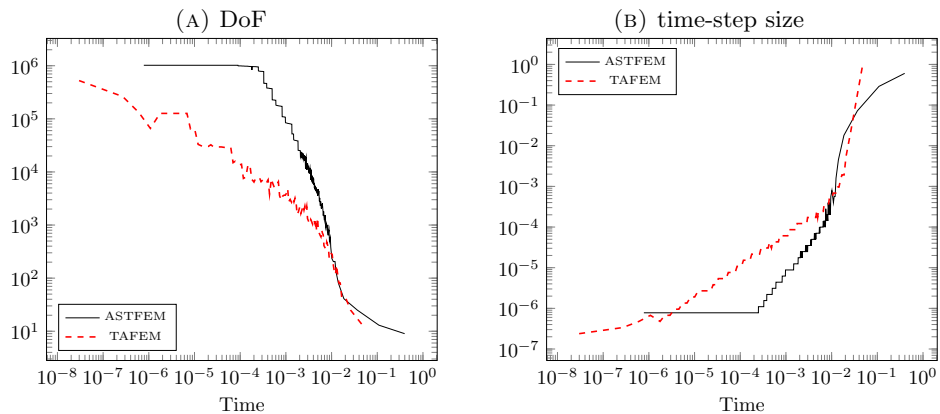
\begin{figure}
\begin{subfigure}[c]{0.22\textwidth}
\subcaption{$t\approx 0$}
\includegraphics[width=1.0\textwidth]{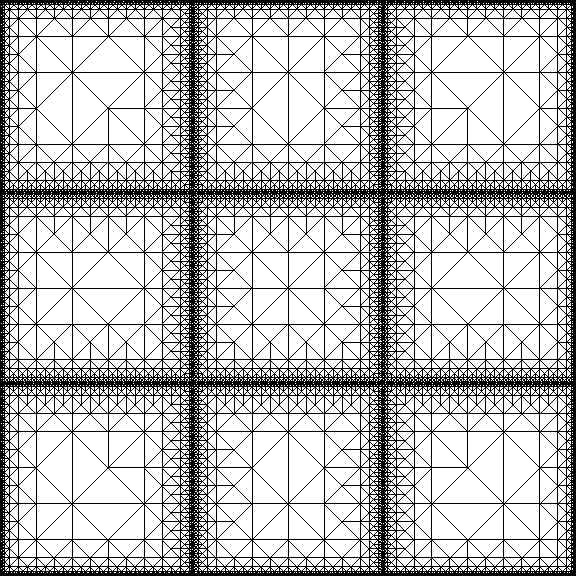}
\end{subfigure}
\quad
\begin{subfigure}[c]{0.22\textwidth}
\subcaption{$t\approx 0.0015$}
\includegraphics[width=1.0\textwidth]{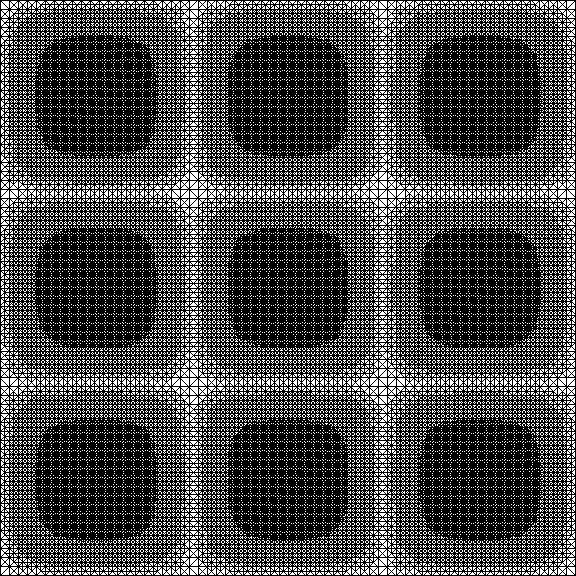}
\end{subfigure}
\quad
\begin{subfigure}[c]{0.22\textwidth}
\subcaption{$t\approx 0.0040$}
\includegraphics[width=1.0\textwidth]{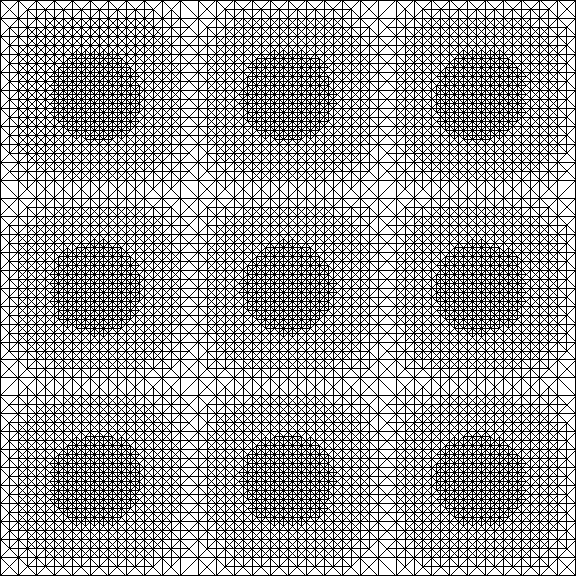}
\end{subfigure}
\quad
\begin{subfigure}[c]{0.22\textwidth}
\subcaption{$t\approx 0.040$}
\includegraphics[width=1.0\textwidth]{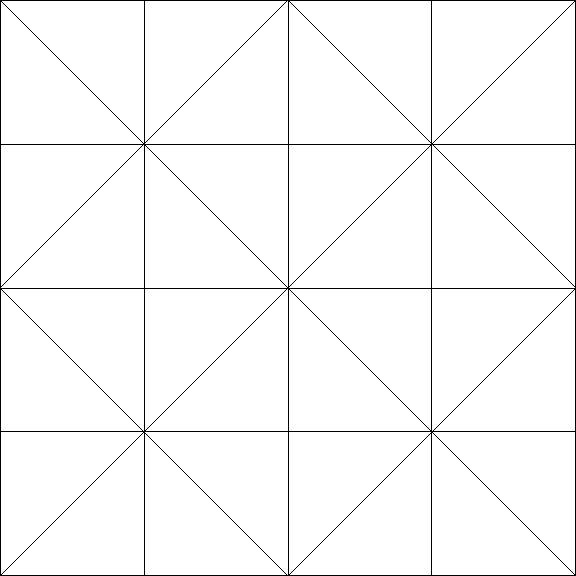}
\end{subfigure}
\begin{subfigure}[c]{0.22\textwidth}
\subcaption{$t\approx 0$}
\includegraphics[width=1.0\textwidth]{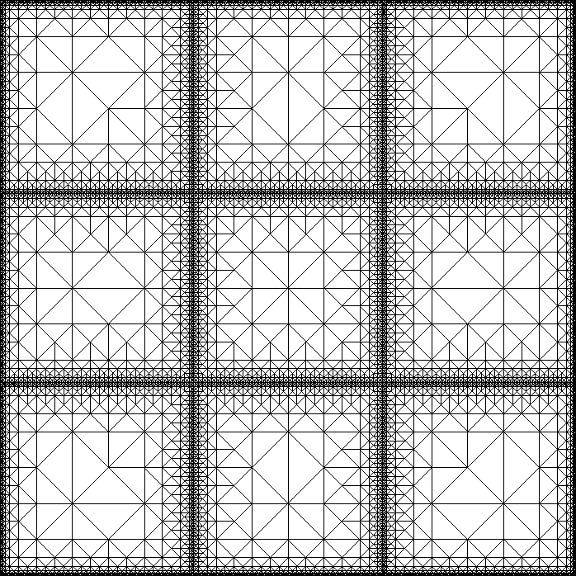}
\end{subfigure}
\quad
\begin{subfigure}[c]{0.22\textwidth}
\subcaption{$t\approx 0.00025$}
\includegraphics[width=1.0\textwidth]{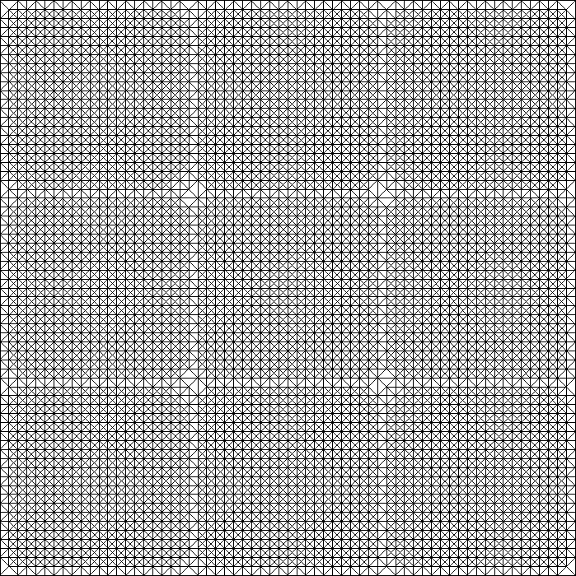}
\end{subfigure}
\quad
\begin{subfigure}[c]{0.22\textwidth}
\subcaption{$t\approx 0.0025$}
\includegraphics[width=1.0\textwidth]{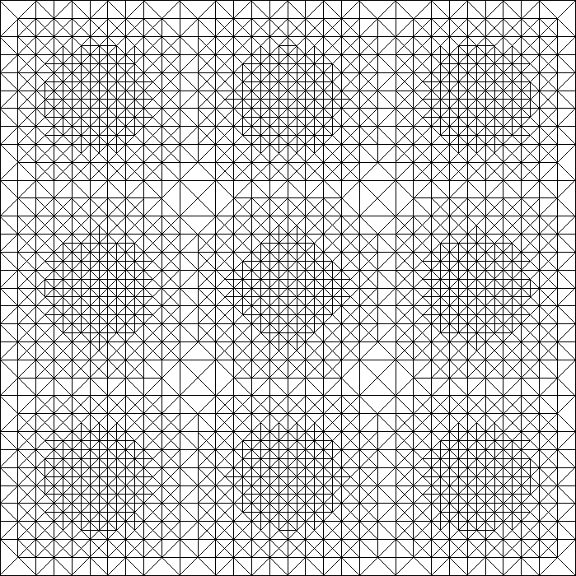}
\end{subfigure}
\quad
\begin{subfigure}[c]{0.22\textwidth}
\subcaption{$t\approx 0.047$}
\includegraphics[width=1.0\textwidth]{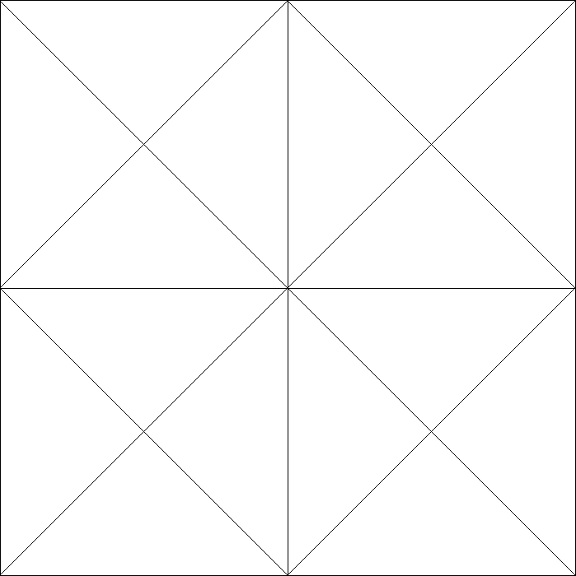}
\end{subfigure}
\caption{Adapted meshes generated with \ASTFEM (A-D) and \TAFEM (E-H) for the rough initial data problem.}
\label{fig:adaptMesh2}
\end{figure}

\newcommand{\etalchar}[1]{$^{#1}$}
\providecommand{\bysame}{\leavevmode\hbox to3em{\hrulefill}\thinspace}
\providecommand{\MR}{\relax\ifhmode\unskip\space\fi MR }
\providecommand{\MRhref}[2]{%
  \href{http://www.ams.org/mathscinet-getitem?mr=#1}{#2}
}
\providecommand{\href}[2]{#2}


\begin{thebibliography}{BBD{\etalchar{+}}16}

\bibitem[AMN09]{AkMaNo:09}
Georgios Akrivis, Charalambos Makridakis, and Ricardo~H. Nochetto,
  \emph{Optimal order a posteriori error estimates for a class of
  {R}unge-{K}utta and {G}alerkin methods}, Numer. Math. \textbf{114} (2009),
  no.~1, 133--160. 

\bibitem[B{\"a}n91]{Bansch:91}
E.~B{\"a}nsch, \emph{Local mesh refinement in 2 and 3 dimensions}, IMPACT
  Comput. Sci. Engrg. \textbf{3} (1991), 181--191.

\bibitem[BBD{\etalchar{+}}16]{DUNE:16}
M.~Blatt, A.~Burchardy, A.~Dedner, Ch. Engwer, J.~Fahlle,
  B.~Flemisch, C.~Gersbacher, C.~Gr\"aser, F.~Gruber,
  C.~Gr\"uninger, D.~Kempf, R.~Kl\"ofkorn, T.~Malkmus, S.~M\"uthing, M.~Nolte, M.~Piatkowski, and O.~Sander, \emph{{The
  Distributed and Unified Numerics Environment (DUNE), Version 2.4}}, Archive
  of Numerical Software \textbf{100} (2016), no.~4, 13--29.

\bibitem[BDD04]{BiDaDe:04}
P.~Binev, W.~Dahmen, and R.~A. DeVore, \emph{Adaptive finite
  element methods with convergence rates}, Numer. Math \textbf{97} (2004),
  219--268.

\bibitem[BDDP02]{BiDaDePe:02}
P.~Binev, W.~Dahmen, R.~DeVore, and P.~Petrushev,
  \emph{Approximation classes for adaptive methods}, Serdica Math. J.
  \textbf{28} (2002), no.~4, 391--416, Dedicated to the memory of Vassil Popov
  on the occasion of his 60th birthday. 

\bibitem[CF04]{ChenFeng:04}
Z.~Chen and J.~Feng, \emph{An adaptive finite element algorithm with
  reliable and efficient error control for linear parabolic problems}, Math.
  Comp. \textbf{73} (2004), 1167--1193.

\bibitem[CKNS08]{CaKrNoSi:08}
J.~M. Cascon, C.~Kreuzer, R.~H. Nochetto, and K.~G.
  Siebert, \emph{Quasi-optimal convergence rate for an adaptive finite element
  method}, SIAM J. Numer. Anal. \textbf{46} (2008), no.~5, 2524--2550.


\bibitem[DK08]{DieningKreuzer:08}
L.~Diening and C.~Kreuzer, \emph{Convergence of an adaptive finite
  element method for the $p$-{L}aplacian equation}, SIAM J. Numer. Anal.
  \textbf{46} (2008), no.~2, 614--638.

\bibitem[DKS16]{DiKrSt:16}
L.~Diening, C.~Kreuzer, and R.~Stevenson, \emph{Instance {O}ptimality
  of the {A}daptive {M}aximum {S}trategy}, Found. Comput. Math. \textbf{16}
  (2016), no.~1, 33--68.

\bibitem[D{\"o}r96]{Dorfler:96}
W.~D{\"o}rfler, \emph{A convergent adaptive algorithm for {P}oisson's
  equation}, SIAM J. Numer. Anal. \textbf{33} (1996), no.~3, 1106--1124.
 

\bibitem[EJ91]{ErikssonJohnson:91}
K.~Eriksson and C.~Johnson, \emph{Adaptive finite element methods for
  parabolic problems {I}: {A} linear model problem}, SIAM J. Numer. Anal.
  \textbf{28} (1991), 43--77.

\bibitem[Eva10]{Evans:10}
L.~C. Evans, \emph{Partial differential equations}, second ed., Graduate
  Studies in Mathematics, vol.~19, American Mathematical Society, Providence,
  RI, 2010. 

\bibitem[ESV]{ErnSmearsVohralik:16}
A.~Ern, I.~Smears, and M.~Vohral{\'{\i}}k, \emph{Guaranteed, locally
  space-time efficient, and polynomial-degree robust a posteriori
  error estimates for high-order discretizations of parabolic
  problems}, submitted.

\bibitem[GM14]{GaspozMorin:14}
F.~D. Gaspoz and P.~Morin, \emph{Approximation classes for adaptive
  higher order finite element approximation}, Math. Comp. \textbf{83} (2014),
  no.~289, 2127--2160. 

\bibitem[GT01]{GilbargTrudinger:01}
D.~Gilbarg and N.~S. Trudinger, \emph{Elliptic partial differential
  equations of second order}, Classics in Mathematics, Springer-Verlag, Berlin,
  2001, Reprint of the 1998 edition.

\bibitem[KMSS12]{KrMoScSi:12}
C.~Kreuzer, C.~A. M{\"o}ller, A.~Schmidt, and K.~G.
  Siebert, \emph{Design and convergence analysis for an adaptive discretization
  of the heat equation}, IMA J. Numer. Anal. \textbf{32} (2012), no.~4,
  1375--1403. 

\bibitem[Kos94]{Kossaczky:94}
I.~Kossaczk\'y, \emph{A recursive approach to local mesh refinement in two and
  three dimensions}, J. Comput. Appl. Math. \textbf{55} (1994), 275--288.

\bibitem[KS11]{KreuzerSiebert:11}
C.~Kreuzer and K.~G. Siebert, \emph{Decay rates of adaptive finite
  elements with {D}\"orfler marking}, Numer. Math. \textbf{117} (2011), no.~4,
  679--716. 

\bibitem[LM06]{LakkisMakridakis:06}
O.~Lakkis and C.~Makridakis, \emph{Elliptic reconstruction and a
  posteriori error estimates for fully discrete linear parabolic problems},
  Math. Comp. \textbf{75} (2006), no.~256, 1627--1658. 

\bibitem[Mau95]{Maubach:95}
J.~M. Maubach, \emph{Local bisection refinement for {$n$}-simplicial grids
  generated by reflection}, SIAM J. Sci. Comput. \textbf{16} (1995), no.~1,
  210--227. 

\bibitem[MN03]{MakridakisNochetto:03}
C.~Makridakis and R.~H. Nochetto, \emph{Elliptic reconstruction
  and a posteriori error estimates for parabolic problems}, SIAM J. Numer.
  Anal. \textbf{41} (2003), no.~4, 1585--1594. 

\bibitem[MN05]{MekchayNochetto:05}
K.~Mekchay and R.~H. Nochetto, \emph{Convergence of adaptive finite
  element methods for general second order linear elliptic {PDE}s}, SIAM J.
  Numer. Anal. \textbf{43} (2005), no.~5, 1803--1827 (electronic). 

\bibitem[MNS00]{MoNoSi:00}
P.~Morin, R.~H. Nochetto, and K.~G. Siebert, \emph{Data
  oscillation and convergence of adaptive {FEM}}, SIAM J. Numer. Anal.
  \textbf{38} (2000), no.~2, 466--488 (electronic). 

\bibitem[MNS02]{MoNoSi:02}
\bysame, \emph{Convergence of adaptive finite element methods}, SIAM Rev.
  \textbf{44} (2002), no.~4, 631--658 (electronic) (2003), Revised reprint of
  ``Data oscillation and convergence of adaptive FEM'' [SIAM J. Numer. Anal.
  {{\bf{3}}8} (2000), no. 2, 466--488 (electronic)].

\bibitem[MSV08]{MoSiVe:08}
P.~Morin, K.~G. Siebert, and A.~Veeser, \emph{A basic convergence
  result for conforming adaptive finite elements}, Math. Models Methods Appl.
  Sci. \textbf{18} (2008), no.~5, 707--737. 

\bibitem[NSV09]{NoSiVe:09}
R.~H. Nochetto, K.~G. Siebert, and A.~Veeser, \emph{Theory of
  adaptive finite element methods: an introduction}, Multiscale, nonlinear and
  adaptive approximation, Springer, Berlin, 2009, pp.~409--542. 

\bibitem[Pic98]{Picasso:98}
M.~Picasso, \emph{Adaptive finite elements for a linear parabolic problem},
  Comput. Methods Appl. Mech. Engrg. \textbf{167} (1998), no.~3-4, 223--237.
  
\bibitem[Sie11]{Siebert:11}
K.~G. Siebert, \emph{A convergence proof for adaptive finite elements
  without lower bound.}, IMA J. Numer. Anal. \textbf{31} (2011), no.~3,
  947--970 (English).

\bibitem[Sme15]{Smears:15}
I.~Smears, \emph{Robust and efficient preconditioners for the discontinuous
  galerkin time-stepping method}, Technical report 1894, University of Oxford,
  2015.

\bibitem[SS05]{SchmidtSiebert:05}
A.~Schmidt and K.~G. Siebert, \emph{Design of adaptive finite element
  software}, Lecture Notes in Computational Science and Engineering, vol.~42,
  Springer-Verlag, Berlin, 2005, The finite element toolbox ALBERTA, With 1
  CD-ROM (Unix/Linux). 

\bibitem[SS09]{SchwabStevenson:09}
C.~Schwab and R.~Stevenson, \emph{Space-time adaptive wavelet methods
  for parabolic evolution problems}, Math. Comp. \textbf{78} (2009), no.~267,
  1293--1318. 

\bibitem[Ste07]{Stevenson:07}
R.~Stevenson, \emph{Optimality of a standard adaptive finite element method},
  Found. Comput. Math. \textbf{7} (2007), no.~2, 245--269. 

\bibitem[Tho06]{Thomee:06}
V.~Thom{\'e}e, \emph{Galerkin finite element methods for parabolic
  problems}, second ed., Springer Series in Computational Mathematics, vol.~25,
  Springer-Verlag, Berlin, 2006. 

\bibitem[Tra97]{Traxler:97}
C.~T. Traxler, \emph{An algorithm for adaptive mesh refinement in $n$
  dimensions}, Computing \textbf{59} (1997), 115--137.

\bibitem[TV16]{TantardiniVeeser:16b}
F.~Tantardini and A.~Veeser, \emph{Quasi-optimality constants for
  parabolic galerkin approximation in space}, Proceedings of the 10th
  {E}uropean {C}onference on {N}umerical {M}athematics and {A}dvanced
  {A}pplications held at the {M}iddle {E}ast {T}echnical {U}niversity,
  {A}nkara, Turkey, {S}eptember 14--18, 2015, To appear in Lecture Notes in
  Computational Science and Engineering, Springer, Heidelberg, 2016.

\bibitem[Ver03]{Verfurth:03}
R.~Verf{\"u}rth, \emph{A posteriori error estimates for finite element
  discretizations of the heat equation}, Calcolo \textbf{40} (2003), no.~3,
  195--212. 

\bibitem[Ver13]{Verfuerth:2013}
\bysame, \emph{A posteriori error estimation techniques for finite element
  methods}, Numerical Mathematics and Scientific Computation, Oxford University
  Press, Oxford, 2013.

\end{thebibliography}
\end{document}